\documentclass[letterpaper,10pt,reqno]{article}
\usepackage{gbmwgc}
\title{Blow-up in manifolds with generalized corners}
\author{Chris Kottke}
\date{\today}
\begin{document}
\maketitle
\begin{abstract}
We construct a functor from the category of manifolds with generalized corners
to the category of complexes of toric monoids, and for every `refinement' of
the complex associated to a manifold, we show there is a unique `blow-up',
i.e., a new manifold mapping to the original one, which satisfies a universal
property and whose complex realizes the refinement.
This was inspired in part by the work of Gillam and Molcho, though we work with
manifolds with generalized corners, as developed by Joyce, which have embedded
boundary faces, for which the appropriate objects (i.e., complexes of monoids)
are simpler than they would be otherwise (i.e., monoidal spaces in the sense of
Kato). 
\end{abstract}

\section{Introduction} \label{S:intro}
The aim of this paper is to provide a self-contained, differential geometric
treatment of boundary blow-up in the category of manifolds with generalized
corners. 
We use the term `blow-up' in a general sense, as described below.

A {\em manifold with generalized corners}, as developed by Joyce \cite{joyce},
is a space $M$ which is locally modeled on the spaces $X_P = \Hom(P; \bbR_+)$
where the $P$ are {\em toric monoids}. 
Such a space has an interior, $M^\circ$, which is a smooth manifold without
boundary, and boundary faces which are themselves manifolds with generalized
corners.
We construct a functor 
\[
	M \mapsto \cP_M,
	\quad (f : M \to N) \mapsto (f_\natural : \cP_M \to \cP_N)
\]
from the category of manifolds with generalized corners and interior b-maps to
the category of {\em monoidal complexes} \cite{KM}.  These are roughly
analogous to simplicial complexes, with toric monoids instead of simplices, and
$\cP_M$ associates a toric monoid to each boundary face of $M$.
A {\em refinement} of the monoidal complex $\cP_M$ is a morphism $\cR \to
\cP_M$ which amounts to giving a consistent subdivision of the monoids in
$\cP_M$ by toric submonoids. 
\begin{thmstar}[Thm.\ \ref{T:main_thm}, \ref{T:blow-up_pullback}]
\mbox{}
\begin{enumerate}
[{\normalfont (i)}]
\item \label{I:blow-up}
For a manifold with generalized corners $M$ and a refinement $\psi :\cR \to
\cP_M$, there exists a unique (up to diffeomorphism) {\bf blow-up}, i.e., a
manifold with generalized corners $[M; \cR]$ and a {\bf blow-down} map, $\beta
: [M; \cR] \to M$, such that $\beta : [M; \cR]^\circ \to M^\circ$ is a
diffeomorphism of interiors and
\[
	\cP_{[M; \cR]} \cong \cR,
	\quad \beta_\natural \cong \psi : \cR \to \cP_M.
\]
\item \label{I:factoring}
The blow-up satisfies the following universal property:
If the morphism $f_\natural : \cP_N \to \cP_M$ of monoidal complexes associated
to an interior b-map $f : N \to M$ factors through $\psi : \cR \to \cP_M$, then
$f$ factors through a unique interior b-map 
\[
	\wt f : N \to [M; \cR].
\]
\item \label{I:pullback}
If $f : N \to M$ is any interior b-map, then 
the pull back of $[M; \cR]$ to $N$ is a blow-up:
\[
	N\times_M [M; \cR] \cong [N; \cR'],
	\quad \cR' = \cP_N \times_{\cP_M} \cR.
\]
\end{enumerate}
\end{thmstar}

Blow-up (in this generalized sense) in the category, $\MC$, of manifolds with
(ordinary) corners was developed in \cite{KM}, a principal result of which was
the resolution of certain `binomial subvarieties' inside a manifold with
corners which arise when taking fiber products, among other situations.
Following that work, Joyce in \cite{joyce} developed the category, $\MGC$, of
manifolds with generalized corners, giving an intrinsic differential
topological characterization of the natural class of objects exemplified by
binomial subvarieties, and showed, among other results, that this category is
closed under suitably transverse fiber products. 

There is also an algebro-geometric theory \cite{GM} due to Gillam and Molcho, 
in which manifolds with corners arise as a natural subcategory of the category,
$\PLDS$, of `positive log differentiable spaces'.  In this formulation, the
`b-' objects (i.e., b-maps, b-tangent bundles, b-differentials and so on)
associated to manifolds with corners as defined by Melrose \cite{melroseaps}
are the natural ones corresponding to a `logarithmic structure' on such a
space, in the sense of \cite{KKMS, kato}. In addition to $\MC$, the category
$\PLDS$ includes $\MGC$ as a subcategory in addition to more general spaces.
Gillam and Molcho extend Kato's resolution of toric singularities \cite{kato}
to this category.
In this formulation, each space $M$ is associated with a 
`monoidal space' $\ol M$, which is a sheaf of toric monoids over $M$. To each
suitable resolution $F \to \ol M$, they prove that there exists an essentially
unique universal smooth space $N \to M$ with $\ol N \to \ol M$ factoring
through $F$.

While Gillam and Molcho's theory is very general, it is also quite abstract and
heavily reliant on high level concepts from algebraic geometry. For this
reason, we present here a short, self-contained, elementary treatment of
blow-up in the category $\MGC$. 
In contrast to Joyce, we require as part of the definition of a manifold with
generalized corners that its boundary faces are embedded. Under this
assumption, the 
monoidal space $\ol M$ may be replaced by the monoidal complex $\cP_M$, a
simpler, essentially combinatorial object carrying the same information.  (See
the discussion in \S\ref{S:commentary} for more on this point.)

\medskip Section \ref{S:background} summarizes the necessary background
material. 
We review toric monoids in \S\ref{S:monoids}, and then devote some detailed
discussion to the model spaces $X_P$ in \S\ref{S:models} before reviewing
manifolds (with generalized corners) in \S\ref{S:MGC}. Most of the ideas in
\S\ref{S:models} and \S\ref{S:MGC} are due to Joyce, though some of our
terminology and notation differs from \cite{joyce}.

We emphasize the structure of a manifold $M$ as a {\em stratified space}, with
strata given by the interiors, $F^\circ$, of boundary faces $F \subset M$. Each
such stratum sits locally inside $M$ as the subset $\set{\star}\times \bbR^l
\subset X_{W(F)} \times \bbR^l$ where $W(F)$ is a fixed monoid (the `conormal
monoid'), and $\star$ is a canonical base point in the model space $X_{W(F)}$,
which forms the fiber of the stratum.

The differential structure of $M$ is encoded by the {\em b-tangent bundle} $\bT
M \to M$, a real vector bundle of rank $\dim(M)$. Each boundary face $F \subset
M$ supports a rank $\codim(F)$ {\em b-normal} subbundle $\bN F \subset \bT M
\to F$, with an underlying trivial bundle $\bM F\to F$ of monoids; in fact $\bN
F = \bM F \otimes_\bbN \bbR$ and $\bM F \cong F \times W(F)^\vee$, where
$W(F)^\vee$ is the `normal monoid' dual to the conormal monoid above.
If $G$ and $F$ are boundary faces with $G \subset F$, then $W(F)^\vee$ identifies
naturally with a face of the monoid $W(G)^\vee$.
An interior b-map $f : N \to M$ gives rise to a map $\bd f_\ast : \bT N \to \bT
M$ of vector bundles which respects these structures; in particular $f$
induces compatible homomorphisms from the normal monoids associated with the
faces of $N$ to those of $M$.

The collections $\cP_M = \set{W(F)^\vee : F \subset M}$  of normal monoids and
these induced homomorphisms are the motivating examples of {\em monoidal
complexes} and their morphisms, which are reviewed in \S\ref{S:mon_cplx}, along
with the notion of refinement.
Finally, in Section~\ref{S:blow-up}, we develop the theory of blow-up, first
for the model spaces in \S\ref{S:blow-up_models} and then for manifolds in
\S\ref{S:global_blow-up}, where we prove parts \eqref{I:blow-up} and
\eqref{I:factoring} of the above theorem. We use Joyce's result on fiber
products in \S\ref{S:blow-up_pullback} to prove part \eqref{I:pullback}, 
and make some concluding remarks in \S\ref{S:commentary}.

\begin{ack}
The author would like to thank Dominic Joyce for his insightful comments 
on an earlier draft of the manuscript.
\end{ack}

\section{Background} \label{S:background}
\subsection{Monoids} \label{S:monoids}
For a general discussion of monoids, see \cite{ogus} or \cite{GM}. We review
here the basic concepts that will be used below.

A {\bf monoid} is a set $P = (P,+,0)$ which is closed under an associative,
commutative, unital binary operation, which we write additively unless
otherwise specified.
A monoid {\bf homomorphism} $f: P \to Q$ is a map such that $f(p+p') = f(p) +
f(p')$ and $f(0) = 0$.
Every abelian group is a monoid, and a monoid homomorphism between groups is
automatically a homomorphism of groups.

A {\bf submonoid} $Q \subset P$ is a subset containing $0$ which is closed
under the binary operation (and thus is also a monoid).
For each submonoid $Q \subset P$, there is a {\bf quotient} monoid $P/Q$ and a
homomorphism
\[
	\pi : P \to P/Q
\]
with the universal property that any homomorphism $h : P \to R$ for which $h(Q)
= \set 0$ factors through a unique homomorphism $\wt h : P/Q \to R$, i.e., $h =
\wt h \circ \pi$. $P/Q$ may be realized as the monoid of equivalence classes
generated by the equivalence relation $p \sim p' \iff p + q = p' + q'$ for some
$q,q' \in Q$.

A {\bf unit} in $P$ is an element $p$ with a (necessarily unique) inverse $q$
such that $p + q = 0$; the units form a submonoid (which is an abelian group)
which we denote by $P^\times$.
%
If $P^\times = \set 0$ then we say $P$ is {\bf sharp}. For any $P$, the monoid
\[
	P^\sharp = P/P^\times
\]
is sharp; it is called the {\bf sharpening} of $P$.

\subsubsection{Localization} \label{S:localization}
To each monoid $P$ we associate an abelian group $P^\gp$ and a homomorphism
\[
	\iota : P \to P^\gp
\]
where $P^\gp$ satisfies the universal property that any homomorphism $h : P \to
G$ to an abelian group factors through a unique homomorphism $\wt h :
P^\gp \to G$, i.e., $h = \wt h \circ \iota$. It follows from the universal
property that $P^\gp$ and $\iota$ are unique up to unique isomorphism.
The {\bf dimension} of $P$ is
\[
	\dim(P) = \rank(P^\gp).
\]

More generally, given any submonoid $S \subset P$, the {\bf localization},
$S^{-1} P$,  of $P$ at $S$ is a monoid with a homomorphism
\[
	\iota : P \to S^{-1}P
\]
such that $\iota(s)$ is a unit for each $s \in S$.
The localization has the the universal property that any homomorphism $h : P\to
Q$ of monoids in which $h(S) \subset Q^\times$ factors through a unique
homomorphism $\wt h : S^{-1} P \to Q$.
In the special case $S = P$, we have
\[
	P^\gp = P^{-1}P.
\]
The localization $S^{-1}P$ may be realized as the set of equivalence
classes of pairs $[p,s]$ with respect to the equivalence relation $(p,s) \sim
(p',s') \iff p + s' + q = p' + s + q$ for some $q \in Q$, with $\iota(p) =
[p,0]$.

\subsubsection{Toric monoids} \label{S:toric}
We say $P$ is {\bf toric} if it is:
\begin{enumerate}
[{\normalfont (T1)}]
\item \label{I:fin_gen}
{\bf finitely generated}, meaning there is a surjective homomorphism $\bbN^n \to P$
for some $n \in \bbN$ (and then $P^\gp$ is a finitely generated abelian group),
\item \label{I:integral}
{\bf integral}, meaning that if $p + r = q + r$ in $P$, then $p = q$; equivalently, the map
$\iota : P \to P^\gp$ is injective,
\item \label{I:tor_free}
{\bf torsion free}, meaning that $np = p + \cdots + p = 0$ implies $p = 0$;
equivalently, $P^\gp$ is a torsion free abelian group, and
\item \label{I:saturated}
{\bf saturated}, meaning that if $p \in P^\gp$ with $np \in P$ for some $n \in
\bbN$, then $p \in P$.
\end{enumerate}
In particular, if $P$ is toric then $P^\gp$ is a {\em lattice} (finitely generated, torsion free
abelian group).

\medskip
\noindent
From this point on, {\em monoid} will mean {\em toric monoid} unless otherwise
specified.
\medskip

\begin{rmk}
Toric monoids as defined here correspond to what Joyce calls `weakly toric' monoids. Joyce
reserves the term `toric' for a sharp toric monoid.
\end{rmk}

As an alternative to the algebraic conditions
(T\ref{I:fin_gen})--(T\ref{I:saturated}), there is a more geometric
characterization of toric monoids which makes them easier to visualize.

\begin{prop}[\cite{joyce}, Prop.\ 3.8]
A toric monoid is equivalent to the intersection of a finitely generated
lattice $L$ with a cone $C \subset V$, where $V = L\otimes_\bbZ \bbR \supset L$
is the associated real vector space, and $C$ is convex, rational and
polyhedral, (i.e., $C$ is the convex hull of a finite number of rays generated
by lattice elements). The monoid $P = C \cap L$ is sharp if and only if $C$
contains no non-trivial subspace.
\label{P:monoid_as_cone}
\end{prop}
In the setting of Proposition~\ref{P:monoid_as_cone}, $P^\gp$ is nothing other
than the lattice $L$ (assuming that the cone $C$ does not lie in any proper
subspace, in which case we can pass to the corresponding sublattice).
Likewise, for each  $S\subset P$, $S^{-1}P$ may be realized as the submonoid of
$L$ generated by $P$ and the minimal sublattice containing $S$.

\subsubsection{Faces and ideals} \label{S:faces}
An {\bf ideal} of $P$ is a proper subset $I \subsetneq P$ such that $i + p \in I$
for all $i \in I$, $p \in P$.
An ideal $I$ is {\bf prime} if $p + q \in I$ implies that either $p \in I$ or $q \in I$.

A {\bf face} of $P$ is a submonoid $S \subset P$ whose complement $P \setminus
S$ is a prime ideal; thus $S$ has the property that if $p + q \notin S$ then
either $p \notin S$ or $q \notin S$.
In the setting of Proposition~\ref{P:monoid_as_cone}, faces of $P = C \cap L$
are precisely the toric monoids given by the intersections $D \cap L$ where $D$
is a face, in the obvious sense, of the polyhedral cone $C$.
Faces (resp.\ prime ideals) are closed under intersection (resp.\ union), and
there is a unique minimal face $P^\times$ (corresponding to the unique maximal
ideal $P \setminus P^\times$) and a unique maximal face $P$ (corresponding to
the minimal prime ideal $\emptyset$).
We write
\[
	S \leq P
\]
for the inclusion of a face, and to denote the (partial) order relation on
faces determined by inclusion.  Note that $T \leq S$ as faces of $P$ if and
only if $T$ is a face of $S$.

The inclusion $S \to P$ generates an exact sequence
\begin{equation}
	S^\gp \to P^\gp \to (P/S)^\gp \cong P^\gp/S^\gp
	\label{E:face_ex_seq}
\end{equation}
of free abelian groups, and the {\bf codimension} of $S$ is
\[
	\codim(S) = \rank(P^\gp/S^\gp) = \dim(P) - \dim(S).
\]

\begin{prop}
For each face $S \leq P$, the exact sequence
\eqref{E:face_ex_seq}
splits (non-canonically), giving an isomorphism
\begin{equation}
	S^{-1}P \cong P/S \times S^\gp.
	\label{E:splitting}
\end{equation}
In particular, $P \cong P^\sharp \times P^\times$.
\label{P:split_ex_seq}
\end{prop}

Since every face $S \leq P$ contains the minimal face $P^\times$, each quotient
$P/S$ is a sharp monoid.

Though it is not standard, we will make use of the following notion in
\S\ref{S:mon_cplx}.
The {\bf interior} of a monoid $P$ is the complement
\[
	P^\circ = \bigcap_{S \lneq P} P \setminus S
\]
of all the proper faces of $P$. It is an ideal, but generally not a prime
ideal.
The interiors of the faces of $P$ determine a partition $P = \bigsqcup_{S \leq
P} S^\circ$.
A homomorphism $f : Q \to P$ is an {\bf interior homomorphism} if
$f(P^\circ) \subset Q^\circ$, i.e., if $f$ does not map $Q$ into any proper
face of $P$.

\subsubsection{Duality} \label{S:duality}
The {\bf dual} of a monoid $P$ is the monoid
\[
	P^\vee = \Hom(P; \bbN).
\]
Since units are preserved by homomorphisms and $\bbN^\times = \set 0$, there is
a natural isomorphism $P^\vee \cong (P^\sharp)^\vee$. Likewise $P^\vee$ is
sharp.
Evaluation $p \mapsto \ev_p \in \Hom(P^\vee; \bbN)$ determines a natural
homomorphism $P \to (P^\vee)^\vee$ with kernel $P^\times$, giving an
isomorphism
\begin{equation}
	P^\sharp \cong (P^\vee)^\vee.
	\label{E:double_dual}
\end{equation}

For each face $S \leq P$, define its {\bf annihilator} by
\[
	S^\perp = \bigcap_{s \in S} \ev_s^{-1}(0)
	= \set{p \in P^\vee : p(s) = 0 \ \forall\,s \in S}
	\leq P^\vee.
\]
This is easily seen to be a face of $P^\vee$ (the subsets $\ev_s^{-1}(0)$ are
prime ideals), and the association $S \mapsto S^\perp$ gives a codimension- and
inclusion-reversing bijection between faces of $P$ and faces of $P^\vee$.
With respect to \eqref{E:double_dual}, we have $S^\sharp \cong
(S^\perp)^\perp$.  There is also a natural isomorphism $(P/S)^\vee \cong
S^\perp$, which is to say that we have dual exact sequences of monoids
\[
\begin{tikzcd}[column sep=small, row sep=tiny]
0 \ar{r} & S \ar{r} & P \ar{r} & P/S \ar{r}& 0 &\text{and}
\\ 0 & \ar{l}P^\vee/S^\perp & \ar{l} P^\vee & \ar{l} S^\perp & \ar{l} 0
\end{tikzcd}
\]
i.e., the second is dual to the first, and vice versa if $P$ is sharp.

\begin{lem}
If $S \leq P$ and $p \in P\setminus S$, then there exists $q \in S^\perp$ such that $q(p) \neq 0$.
\label{L:nontriv_dual}
\end{lem}
\begin{proof}
For each $s \in S^\perp$, $s^{-1}(0)$ is a prime ideal, which is to say the
complement of some face $T$ with $S\leq T$ by necessity. Then
$(S^\perp)^{-1}(0) = \bigcup_{s\in S^\perp} s^{-1}(0)$ is the complement of a
face which must be $S$ by the property $(S^\perp)^\perp = S^\sharp$. By
hypothesis $p \notin S = (S^\perp)^{-1}(0)$, so there is some $q \in S^\perp$
with $q(p) \neq 0$.
\end{proof}

\subsubsection{Examples} \label{S:monoid_examples}
A basic example is $P = \bbN^n \times \bbZ^m$.
Here $P^\times \cong \set 0 \times \bbZ^k$, so $P$ is sharp if and only if $m =
0$, and $P^\sharp \cong \bbN^n$, while $P^\gp = \bbZ^{n+m}$.  The faces of $P$
are the sets $\set{(a_1,\ldots,a_n,b_1,\ldots,b_m) : a_i = 0, i \in I}$ for
various $I \subset \set{1,\ldots,n}$.

A {\bf freely generated} monoid is isomorphic to $\bbN^n$ for some $n$; more
generally we say a monoid $P$ is {\bf smooth} if $P^\sharp$ is freely
generated; in this case $P$ is isomorphic to  $\bbN^n\times \bbZ^m$, where $n+m
= \dim(P)$ and $m = \rank(P^\times)$.

By the property \eqref{I:fin_gen}, every sharp monoid may be presented
(non-canonically) as a submonoid of $\bbN^n$ by choosing generators
$p_1,\ldots,p_n \in P$ and imposing finitely many generating relations of the
form
\begin{equation}
	\sum_{i=1}^n a_i^j p_i = \sum_{i=1}^n b_i^j p_i,
	\quad j = 1,\ldots,k,
	\quad a_i^j,b_i^j \in \bbN.
	\label{E:gen_relns}
\end{equation}
Using $P \cong P^\sharp \times P^\times$, any monoid may then be presented as a
submonoid of $\bbN^n\times \bbZ^m$ by choosing generators
\begin{equation}
	\set{p_1,\ldots,p_n,\pm q_1,\ldots, \pm q_m},
	\quad \set{p_i} \in P\setminus P^\times,
	\quad P^\times \cong \bbZ\pair{q_1,\ldots,q_m},
	\label{E:gens_nonsharp}
\end{equation}
with generating relations on the $p_i$ as above.

For several examples of non-toric monoids, see Example 3.2 in \cite{joyce}.
Among these we highlight one which plays an important role below: consider the
multiplicative monoid
\begin{equation}
	\bbR_+ = ([0,\infty),\cdot,1).
	\label{E:Rplus}
\end{equation}
In the first place, $\bbR_+$ is not finitely generated, and the identity
$0\cdot a = 0$ for all $a \in \bbR_+$ implies that $\bbR_+^\gp = \set 0$, so
that $\bbR_+$ is not integral. Moreover, $\bbR_+^\times = (0,\infty)$,
so that $\bbR_+$ is not sharp.

\subsection{Model spaces} \label{S:models}
To each monoid $P$, we associate the space
\[
	X_P = \Hom(P; \bbR_+),
\]
with $\bbR_+$ as in \eqref{E:Rplus}.
We distinguish a set of {\bf algebraic functions} on $X_P$,
namely, for each $p \in P$, let
\[
	x_p : X_P \to \bbR_+,
	\quad x_p(x) = x(p).
\]
Then $X_P$ is given the weakest topology for which these algebraic functions
are continuous.

Since homomorphisms preserve units, if $p \in P^\times$, it follows that $x(p)
\in (0,\infty)$ for all $x \in X_P$; equivalently, $x_p$ is a strictly positive
function.
If $P$ is sharp, then there is a distinguished point $\star \in X_P$ given by
the constant homomorphism $\star(p) = 0$ for all $p$, and each $x_p$ vanishes
at $\star$.

\begin{rmk}
We will see below that the $x_p$ play the role of {\em coordinates} on $X_P$;
for this reason we use the same letter $x$ to denote both points of $X_P$ and
(with subscripts) algebraic functions. No confusion should arise from this
convention.
\end{rmk}

The algebraic functions generate a smooth structure on $X_P$ in the following
sense. We say that a function $f : O \subset X_P \to \bbR$ defined on an open set
is a {\bf smooth function} if there exist $p_1,\ldots,p_n \in P$ and $g \in \C(W;
\bbR)$, where $W = x_{p_1}\times \cdots \times x_{p_n}(O) \subset \bbR^n_+$,
such that
\begin{equation}
	f = g\circ (x_{p_1}\times \cdots \times x_{p_n}) = g(x_{p_1},\ldots,x_{p_n}).
	\label{E:smooth_function_local}
\end{equation}
The smooth functions form a sheaf of $\bbR$-algebras on $X_P$ which we denote
by
\[
	\C_{X_P}(\bullet) =  \C(\bullet;\bbR).
\]

Lying `in between' the algebraic and smooth functions is the following notion.
We say that a function $b : O\subset X_P \to \bbR_+$ defined on an open set is
a {\bf b-function} if it is locally algebraic up to multiplication by a smooth,
strictly positive function; that is, for all $x \in O$, there is a possibly smaller
neighborhood $x \in O' \subset O$ on which
\begin{equation}
	b\rst O' = h\, x_p,
	\quad \text{for some}\ p \in P,
	\ h \in \C\big(O'; (0,\infty)\big).
	\label{E:b_function_local}
\end{equation}
Note that $h$ and $p$ are not uniquely determined since there may be many $q
\in P$ for which $x_q \rst O'$ is strictly positive.
The b-functions form a sheaf of (non-toric) monoids which we denote by
\[
	\B_{X_P}(\bullet) = \B(\bullet; \bbR_+),
\]
where $\B(O; \bbR_+)$ denotes the set of b-functions on $O$.
%
%
It is often convenient to allow the constant function $0$ (which is neither
algebraic nor a b-function); we denote the resulting sheaf of monoids by
\[
	\bB_{X_P} = \B_{X_P} \sqcup 0,
\]
with the obvious multiplication identity $0 \cdot b = 0$.

\begin{rmk}
There is an injective morphism $\B_{X_P} \to \C_{X_P}$ of sheaves.  If we
consider not the sheaf $\C(\bullet; \bbR)$ but rather $\C(\bullet; \bbR_+)$ as
the {\em structure sheaf} of the space (see \cite{GM}), then we still have a
morphism $\B_{X_P} \to \C_{X_P}$, but this has the property that
$\B_{X_P}^\times \to (\C_{X_P})^\times$ is an isomorphism. In the language of
log geometry \cite{KKMS, kato}, $\B_{X_P}$ is a {\em logarithmic structure} on
$X_P$.
\end{rmk}

The {\bf interior} of $X_P$ is the subspace
\[
	X^\circ_P = \Hom\big(P; (0,\infty)\big)
\]
of monoid homomorphisms to $(0,\infty)$.
Then $X^\circ_P \subset X_P$ is a dense open set.
In fact, $(0,\infty)$ is a group, so by the universal property of $P^\gp$ we
have
\begin{equation}
	X^\circ_P = X_{P^\gp} = \Hom\big(P^\gp;(0,\infty)\big) \cong (0,\infty)^{\dim(P)}.
	\label{E:intX_P}
\end{equation}
Smooth functions on $X^\circ_P$ as defined above coincide with the usual
notion of smooth functions on the manifold $(0,\infty)^{\dim(P)}$; thus every
$X_P$ has an interior which is diffeomorphic to $\bbR^{\dim(P)}$ via
$\log : (0,\infty)\cong \bbR$.

\begin{ex}
Every monoid homomorphism $x : \bbN \to \bbR_+$ is of the form $x(n) = a^n$ for
some $a = x(1) \in \bbR_+$, and likewise $x : \bbZ \to \bbR_+$ must be of the
same form for $a \in (0,\infty)$. Since the functor $\Hom(\bullet; \bbR_+)$
preserves finite products, we have
\[
	X_{\bbN^n\times \bbZ^m} = \bbR_+^n \times (0,\infty)^m \cong \bbR_+^n\times \bbR^m,
\]
which are the model spaces for manifolds with corners.
\label{X:X_NZ}
\end{ex}

\subsubsection{b-maps} \label{S:bmaps}
While $P \mapsto X_P$ is a contravariant functor from monoids to spaces, we
want to consider more general maps $X_P \to X_Q$ than those which arise from
homomorphisms $Q \to P$.
We say that a map $f : O \subset X_P \to X_Q$ defined on an open set is a {\bf
b-map} if, for every $q \in Q$, there exists some $p \in P$ and $h \in \C\big(O;
(0,\infty)\big)$ such that
\begin{equation}
	f^\ast x_q = h\,x_p
	\quad \text{or}\quad f^\ast x_q = 0.
	\label{E:b-map}
\end{equation}
If the second case never occurs, we say $f$ is an {\bf interior} b-map.
It follows that a b-map $f$ is a {\bf smooth map}, in the sense that
$f^\ast\C_{X_Q} \to \C_O$; however this notion of smoothness, without the
additional requirement \eqref{E:b-map}, turns out to be too weak to be very
useful. Thus by a {\bf diffeomorphism} $f : O\subset X_P \to U \subset X_Q$, we
mean an invertible interior b-map whose inverse, $f^{-1}$, is an interior
b-map.

The following are easy consequences of the definitions.
\begin{prop}
\mbox{}
\begin{enumerate}
[{\normalfont (i)}]
\item
A b-function $f : O \subset X_P \to \bbR_+$ is equivalent to an interior b-map
to $\bbR_+$, where the latter is considered as the model space $\bbR_+ =
X_{\bbN}$.
\item
A b-map $f : O \subset X_P \to X_Q$ is interior if and only if $f(O\cap
X^\circ_P) \subset X^\circ_Q$.
\item
Every homomorphism $Q \to P$ induces an interior b-map $X_P \to X_Q$.
\item
$f : O \subset X_P \to X_Q$ is an interior b-map if and only if it pulls back b-functions:
\[
	f^\ast \B_{X_Q} \to \B_O.
\]
\item
More generally, $f : O \subset X_P \to X_Q$ is a (not necessarily interior) b-map if and only if
\[
	f^\ast \bB_{X_Q} \to \bB_O.
\]
\end{enumerate}
\label{P:b-maps}
\end{prop}
\medskip
\noindent
From this point on, {\em map} will mean {\em interior b-map} unless otherwise
specified.
\medskip

\begin{rmk}
The definitions above differ from \cite{joyce}. Joyce calls b-maps and b-functions
simply `smooth', distinguishing smooth functions with target $\bbR_+$ from
those with target $\bbR$ or $(0,\infty)$.
\end{rmk}

While Example~\ref{X:X_NZ} gives the basic model space for a smooth
monoid, we can embed a general $X_P$ into some $\bbR_+^n\times \bbR^m$ by
choosing generators and relations. The following is a straightforward
consequence of the definitions.

\begin{prop}[\cite{joyce}, Prop.\ 3.14]
Let $P$ be a monoid, with generators \eqref{E:gens_nonsharp} and generating
relations \eqref{E:gen_relns}.  Write $x_i = x_{p_i} : X_P \to \bbR_+$ and $y_i
= x_{q_i} : X_P \to (0,\infty)$. Then the map
\[
	x_1\times \cdots \times x_n \times y_1\times\cdots \times y_m : X_P
	\to \bbR_+^n\times (0,\infty)^m
\]
is a diffeomorphism onto its image
\[
	\set{(x_1,\ldots,x_n,y_1,\ldots,y_m) :
	\textstyle\prod_i x_i^{a_i^j} = \prod_i x_i^{b_i^j},\; j = 1,\ldots,k}
	\subset \bbR_+^n\times (0,\infty)^m.
\]
\label{P:nonsmooth_model}
\end{prop}

We refer to $(x,y) = (x_1,\ldots,x_n,y_1,\ldots,y_m)$ as {\bf
coordinates} on $X_P$, and emphasize the fact that they depend on a choice
of generators for $P$.

Suppose $(x,y)$
and $(x',y')$
are coordinates on $X_P$ and $X_Q$, respectively, and $f : O \subset X_P \to
X_Q$ is a b-map. Then
\[
	f^\ast x'_j = h'_j(x,y)\,
	  \big(\textstyle\prod_i x_i^{\mu_{ij}}\, \prod_i y_i^{\nu_{ij}}\big)
	= h_j(x,y)\,\prod_i x_i^{\mu_{ij}},
\]
for some $\mu_{ij} \in \bbN$, $\nu_{ij} \in \bbZ$ and smooth $h'_j > 0$, and
since $y_i$ and $y_i^{-1}$ are already smooth positive functions, we absorb
these into $h_j > 0$.
Since the $y'_j$ are also smooth and positive, $f^\ast y'_j$ is just a smooth
positive function $g_j(x,y)$.
It follows that $f$ has the coordinate form
\begin{gather}
	f : (x,y) \mapsto \big(h(x,y)\, x^\mu, g(x,y)\big) = (x',y'),
	\label{E:local_b_map}
	\\ h \in C^\infty(O; (0,\infty)^{n'}),
	\ g \in C^\infty(O; (0,\infty)^{m'}),
	\ \mu \in \Mat(n\times n'; \bbN),
	\notag
\end{gather}
where we use the obvious vector notation.
Note that, with this convention, exponentiation and matrix multiplication are
related by $(x^\mu)^\nu = x^{\mu\nu}$.

\subsubsection{Boundary faces and support} \label{S:model_faces}
For each face $S \leq P$, the inclusion $S \hookrightarrow P$ induces a
surjective map $X_P \to X_S$.
Conversely, $X_S$ is embedded in $X_P$ by a canonical section of this map;
indeed, every $x \in X_S$ has a lift $\wt x \in X_P$ defined by
\begin{equation}
	\wt x (p) = \begin{cases} x(p) & p \in S \\ 0 & p \in P \setminus S. \end{cases}
	\label{E:boundary_inclusion}
\end{equation}
We identify $X_S$ with its image in $X_P$ under $x \mapsto \wt x$, and refer to
it as a {\bf boundary face}.
This embedding of $X_S$ into $X_P$ is a non-interior b-map.
There is an inclusion-preserving bijection between faces of $P$ and boundary
faces of $X_P$.

Every $x \in X_P$ lies in the interior of a unique boundary face; indeed,
$x^{-1}(0)$ is a prime ideal in $P$, and the {\bf support} of $x$ is the face
\[
	\supp(x) = S \leq P
	\ \text{such that}\ x^{-1}(0) = P \setminus S.
\]
Then $x \in X^\circ_{\supp(x)}$, and as a result we obtain a stratification
\begin{equation}
	X_P = \bigsqcup_{S \leq P} X^\circ_S,
	\quad \ol{X^\circ_S} = X_S = \bigsqcup_{T \leq S} X^\circ_T,
	\label{E:local_stratn}
\end{equation}
where $\ol{X^\circ_S}$ denotes the closure of $X^\circ_S$ in $X_P$.

Note that the monoid $P$ is not itself a local diffeomorphism invariant; it may
certainly happen that an open set $O \subset X_P$ is diffeomorphic to an open
set $U \subset X_Q$ even if $P \not \cong Q$ (for example, $X^\circ_P \cong
X^\circ_Q$ whenever $\dim(P) = \dim(Q)$). Likewise, $\supp(x)$ is not a local
invariant, though it turns out that the quotient monoid $P/\supp(x)$ {\em is}
invariant.
Denote by $\B_x$ the stalk of $\B$ at $x$. It is a generally non-toric monoid.

\begin{lem}[\cite{joyce}, \S3.4]
For $x \in X_P$, the sharpening $\B_x^\sharp$ is a toric monoid, and there is a
canonical isomorphism
\begin{equation}
	\B_x^\sharp \cong P/\supp(x).
	\label{E:B_sharp}
\end{equation}
\label{L:B_sharp}
\end{lem}
\fixthmeq
\begin{proof}
The units $\B_x^\times$ are the germs $[b]$ of b-functions such that $b(x) >
0$, and $\B_x^\sharp = \B_x/\B_x^\times$. The homomorphism $P \to \B_x^\sharp$,
$p \mapsto [x_p]$ is surjective, since every $b$ has the local form $b =
h\,x_p$ for some $p$ and $h > 0$, and hence $[b] \equiv [x_p] \mod \B_x^\times$.
Moreover, the kernel of this homomorphism is precisely $\supp(x) \leq P$, since
$x_p(x) = x(p) > 0$ if and only if $p \in \supp(x)$.
\end{proof}

In light of this result, we define the {\bf codimension} of $x \in X_P$ by
\[
	\codim(x) = \dim(\B_x^\sharp) = \codim(\supp(x)).
\]
By Proposition~\ref{P:b-maps}, a local diffeomorphism $f : O \subset X_P \to U
\subset X_Q$ induces isomorphisms $\B_x \cong \B_{f(x)}$ and $\B^\sharp_x \cong
\B^\sharp_{f(x)}$. It follows that $\B_x^\sharp$ and $\codim(x)$ are local
diffeomorphism invariants.

\subsubsection{Localization and normal models} \label{S:model_normal_models}
For each face $S \leq P$, the localization $\iota : P \to S^{-1} P$ induces
a map
\begin{equation}
	\iota^\ast : X_{S^{-1}P} \to X_P
	\label{E:iota_ast}
\end{equation}
of spaces.
In the case $S = P$, this is the inclusion of $X_{P^\gp} = X^\circ_P$ into $X_P$.

\begin{prop}
The map \eqref{E:iota_ast} is a diffeomorphism onto its image,
which is a dense open set in $X_P$ containing $X^\circ_S$.
\label{P:U_S}
\end{prop}
\begin{proof}
The image of $X_{S^{-1}P} = \Hom(S^{-1}P; \bbR_+)$ in
$X_P$ consists of those $x\in \Hom(P; \bbR_+)$ such that $x(s) \neq 0$ for
all $s \in S$ (which necessarily includes all $x \in X^\circ_S$).
Conversely, any such $x \in X_P$ extends to a unique $\wt x \in X_{S^{-1}P}$
via $\wt x(-s) = x(s)^{-1}$.  Thus $\iota^\ast$ is a bijection onto its image, and
it is straightforward to verify that $\iota^\ast$ and its inverse are interior
b-maps.

To see that $\iota^\ast(X_{S^{-1}P})$ is open and dense, we first restate the
characterization of $\iota^\ast(X_{S^{-1}P})$ above as the condition that $x_s
> 0$ for all $s \in S$.  Then we can write $\iota^\ast(X_{S^{-1}P})$ as the
finite intersection of the open dense sets $x_s^{-1}((0,\infty))$, for a finite
set of generators for $S$.
\end{proof}

From this point on, we identify $X_{S^{-1}P}$ with its image in $X_P$.
Using \eqref{E:splitting} and $X_{S^\gp} = X^\circ_S$ we have
\[
	X_{S^{-1}P}\cong X_{P/S} \times X^\circ_S,
\]
and $X_S^\circ$ sits in $X_{S^{-1}P}$ as the subset $\set{\star} \times X_S^\circ$.
We call $X_{P/S}$ the {\bf normal model} for $X^\circ_S$ in $X_P$, and refer to $P/S$
as the {\bf conormal monoid}, for reasons which will become clear later.
Applying this to $\supp(x)$ and using Lemma~\ref{L:B_sharp} we obtain

\begin{cor}
Each $x \in X_P$ has an open neighborhood diffeomorphic to
\[
	X_{W(x)} \times \bbR^l,
\]
where $W(x) = \B_x^\sharp \cong P/\supp(x)$ and $l = \dim(P) - \codim(x)$,
in which $\supp(x)$ is represented by $\set{\star} \times \bbR^l$ and $x$
is identified with the point $(\star, 0)$.

The conormal monoid $W(x)$, the normal model space $X_{W(x)}$, and the numbers $\dim(P)$ and
$\codim(x)$ are all local diffeomorphism invariants.
\label{C:normal_nbhd}
\end{cor}

\subsection{Manifolds with generalized corners} \label{S:MGC}
A {\bf manifold with generalized corners} is a second
countable Hausdorff space $M$ which is locally diffeomorphic to open sets $O
\subset X_P$ for various $P$.
More precisely, $M$ is equipped with a {\em maximal atlas} $\set{O'_a}_{a \in
A}$ of open sets $O'_a \subset M$ and homeomorphisms (aka {\em charts}) $\phi_a
: O'_a \cong O_a \subset X_{P(a)}$ such that the {\em transition functions}
\[
	\phi_a \phi^{-1}_b : \phi_b(O'_a \cap O'_b) \subset X_{P(b)} \to \phi_a(O'_a \cap O'_b) \subset X_{P(a)}
\]
are diffeomorphisms.
Since $\dim(P)$ is a local diffeomorphism invariant, it follows that the $P(a)$
for each connected component of $M$ have the same dimension, which we define to
be the {\bf dimension} of that component; for simplicity we may assume that $M$
is connected and then
\[
	\dim(M) = \dim\big(P(a)\big) \text{ for all $a$}.
\]

A function on $M$ is a {\bf smooth function} (resp.\ {\bf b-function}) if its
compositions with the charts of $M$ are smooth (resp.\ b-). These form sheaves
of $\bbR$-algebras (resp.\ monoids) which we denote by $\C_M$ (resp.\ $\B_M$).
It suffices to check these conditions with respect to a single atlas (i.e., a
single open cover by charts) on $M$.
Likewise, a {\bf b-map} (resp.\ interior b-map) $f : M \to N$ between manifolds
with generalized corners is a continuous map all of whose compositions with the
charts of $M$ and $N$ are b-maps (resp.\ interior b-maps).
Equivalently, $f$ is a b-map if and only if $f^\ast(\B_{N}) \subset
\bB_{M} = \B_{M} \sqcup 0$ and is interior if and only if
$f^\ast(\B_{N}) \subset \B_{M}$.

A manifold with generalized corners has a well-defined stratification by
codimension:
\begin{equation}
	M = \bigsqcup_{0 \leq l \leq \dim(M)} S^l(M),
	\quad S^l(M) = \set{x \in M : \codim(x) = l}.
	\label{E:codim_stratn}
\end{equation}
Indeed, it follows from Corollary~\ref{C:normal_nbhd} that the $S^l(M)$ are smooth
open manifolds of dimension $\dim(M) - l$ and from \eqref{E:local_stratn} that
$\ol{S^l(M)} = \bigsqcup_{k \geq l} S^k(M)$.

According to Joyce, a {\bf boundary face} of $M$ with codimension $l$ is a
connected component, $F$, of the set $C_l(M)$ of pairs $(x,\gamma)$, where $x
\in M$ and $\gamma$ is a consistent choice of connected component of $S^l(M)
\cap U$ as $U$ ranges over sufficiently small neighborhoods of $x$ ($x$ itself
need not be in $S^l(M)$).
In local charts, $F$ is identified with various $X_S \leq X_P$ with $\codim(S)
= l$.
Codimension 0 boundary faces are the connected components of $M$.

The {\bf interior} of a boundary face $F$ is the set, $F^\circ$, of $(x,\gamma)
\in F$ such that $x \in S^l(M)$. In fact, it follows from
Corollary~\ref{C:normal_nbhd} that each $x \in S^l(M)$ has a unique $\gamma$,
so $F^\circ$ is simply a connected component of $S^l(M)$. 
Locally, $F^\circ$ is identified with $X^\circ_S$ for the various $X_S$.
A boundary face $F$ inherits from $M$ the structure of a manifold with
generalized corners of dimension $\dim(M) - l$, and $F^\circ$ is an open manifold of
the same dimension.

It is often the case that the boundary faces of $M$ are {\em embedded}, i.e.,
the map $F \to M$, $(x,\gamma)\mapsto x$ is injective. (It suffices for this to
hold for boundary faces of codimension 1.) In this case a boundary face of $M$
is simply the closure of a connected component of $S^l(M)$ in $M$. 
It is convenient to require this as part of the definition of a manifold with
generalized corners, which we do from now on. We will come back to this point
in \S\ref{S:commentary}.

\begin{rmk}
The definition of manifolds with (ordinary) corners and b-maps
\cite{melroseaps} is recovered by requiring all the model monoids $P(a)$ to be
smooth.
\end{rmk}

\medskip\noindent
From this point on, {\em manifold} will mean {\em connected manifold with
generalized corners and embedded boundary faces} and {\em map} will mean {\em
interior b-map}, unless otherwise specified.

\subsubsection{Manifolds as stratified spaces} \label{S:mflds_as_stratified}
We denote the set of boundary faces of $M$ by $\cF(M) = \bigsqcup_{0 \leq l
\leq \dim(M)} \cF_l(M)$, where $\cF_l(M)$ is the set of boundary faces of
codimension $l$, and we use the notation
\[
	G \leq F \iff F,G \in \cF(M), \text{ with } G \subset F
\]
to denote the partial order relation on boundary faces.
We consider the stratification of $M$ by boundary faces (which is finer than
the stratification by codimension above):
\begin{equation}
	M = \bigsqcup_{F \leq M} F^\circ,
	\quad \ol{F^\circ} = F = \bigsqcup_{G \leq F} G^\circ.
	\label{E:boundary_stratn}
\end{equation}

Corollary~\ref{C:normal_nbhd} implies that each $x \in F^\circ \leq M$ has a
neighborhood diffeomorphic to $X_{W(x)} \times \bbR^{\dim(F)}$ for some $W(x)$.
By diffeomorphism invariance of $W(x)$ and the assumption that $F$ is
connected, it follows that $W(x) = W(F)$ is independent of $x \in F^\circ$,
i.e., that all $x \in F^\circ$ have the same normal model $X_{W(F)}$.
\begin{rmk}
In the language of stratified spaces \cite{pflaum}, we say that $M$, equipped
with the stratification \eqref{E:boundary_stratn} is a {\em topologically
locally trivial} stratified space, with each stratum $F^\circ$ having a fixed
{\em typical fiber} $X_{W(F)}$.
The {\em depth} of a point $x \in M$ coincides with its codimension, and
\eqref{E:codim_stratn} is the depth stratification of $M$.
It satisfies the Mather conditions, and by Proposition~\ref{P:nonsmooth_model}
admits a {\em smooth structure} (in the sense of stratified spaces), with
respect to which the sheaf of smooth functions is equivalent to the one defined
above.
\end{rmk}

\begin{prop}
Let $f : M \to N$ be a b-map (not necessarily interior). Then for each $F \leq
M$, there is a unique $G \leq N$ with the properties that $f(F^\circ) \subset G^\circ$ and
$f \rst F : F \to G$ is an interior b-map. In particular, every b-map is a
morphism of stratified spaces.
\label{P:b-maps_stratified}
\end{prop}
\begin{proof}
First consider the local case of a b-map $f : O \subset X_P \to X_Q$. The set
$q \in Q$ such that $f^\ast (x_q) = 0$ is a prime ideal $Q \setminus T$ for
some $T \leq Q$, and it follows that $f(O) \subset X_T$. Since $f^\ast(x_t) \neq
0$ for $t \in T$, it follows that $f : O \to X_T$ is interior.

Now suppose $O \subset X_P$ is a chart for $O' \subset M$ with $F \cap O \neq
\emptyset$. The stratification \eqref{E:boundary_stratn} restricted to $O'$ is
identified with the intersections of $O$ with the strata $X_P = \bigsqcup_{S
\leq P} X_S^\circ$; in particular $F^\circ \cap O' \cong O \cap X^\circ_S$ for
some $S$.  Applying the local result to $f \rst (O \cap X^\circ_S)$, it follows
that $f$ maps $F^\circ \cap O'$ into $G^\circ \cap U'$ for some $G \leq N$.
Since $F^\circ$ and $G^\circ$ are smooth connected manifolds, $f : F^\circ \to
G^\circ$ globally.  It remains to show that $E \leq F$ implies $H \leq G$,
where $H$ is the unique boundary face such that $f : E^\circ \to H^\circ$, but
since this holds locally it must hold globally as well.
\end{proof}

We associate to each b-map $f : M \to N$ the map of sets
\begin{equation}
	\cF(f) : \cF(M) \to \cF(N),
	\quad \cF(f)(F) = G \text{ such that $f(F^\circ) \subset G^\circ$}.
	\label{E:f_sharp}
\end{equation}

\subsubsection{Tangent bundle} \label{S:tangent}
The usual notion of tangent bundle, defined via derivations on germs of smooth
functions, is not particularly useful in the setting of corners, generalized or
not.
Melrose introduced the b-tangent bundle on a manifold with corners
(\cite{melrose},\cite{melroseaps}), and in \cite{joyce} Joyce extended this to
the setting of generalized corners.

Let $x \in M$ and consider the stalk $\C_x$ of $\C_M$ at $x$. It has two monoid
structures, with respect to addition and multiplication. There are two natural
monoid homomorphisms
\[
	\exp : (\C_x,+,[0]) \to \B_x,
	\quad \inc : \B_x \to (\C_x,\cdot,[1]),
\]
given by exponentiation and inclusion, respectively.
The {\bf b-tangent space} at $x$ is the real vector space, $\bT_x M$, of pairs
\begin{align}
	(v,v') &\in \Der(\C_x; \bbR)\times \Hom(\B_x; \bbR)
	\quad \text{such that} \notag
	\\ v(\inc[b]) &= b(x)\,v'([b]),
	\quad \text{for all $[b] \in \B_x$}.
	\label{E:bT_condition}
\end{align}
Here $v$ is a derivation on $\C_x$, $v'$ is a monoid homomorphism from $\B_x$
to $(\bbR,+)$ and \eqref{E:bT_condition} is consistent with these structures
and with the $\bbR$ vector space structures on $\Der(\C_x;\bbR)$ and
$\Hom(\B_x;\bbR)$.
On one hand, since $v$ is a derivation, $v([\exp f]) = \exp f(x)\,v([f])$,
while on the other hand from \eqref{E:bT_condition}, $v([\exp f]) =
v(\inc(\exp[f])) = \exp f(x)\,v'(\exp[f])$. It follows that $v$ and $v'$ also
satisfy the condition
\begin{equation}
	v([f]) = v'(\exp[f])
	\quad \text{for all $[f] \in \C_x$}.
	\label{E:v_from_vprime}
\end{equation}

\begin{prop}[\cite{joyce}, Example 3.41]
In the model space $X_P$, there is an isomorphism of vector spaces from
$\Hom(P; \bbR)$ to $\bT_x X_P$ for any $x \in X_P$, given by
\begin{equation}
	\Hom(P; \bbR) \ni \alpha
	\mapsto v = \sum_{p \in P} \alpha(p) x_p \pa_{x_p},
	\label{E:bT_local}
\end{equation}
with $v' \in \Hom(\B_x; \bbR)$ determined from $v$ by
$v'([x_p]) = \alpha(p)$ and \eqref{E:v_from_vprime}.

In particular, $\bT_x X_P$ is a finite dimensional vector space with $\dim(\bT_x X_P) = \dim(P)$.
\label{P:bT_local}
\end{prop}

\noindent The formal sum \eqref{E:bT_local}
and the characterization of $v'$
mean that if $f \in \C_{X_P}(O)$ and $b \in \B_{X_P}(O)$ are given as in
\eqref{E:smooth_function_local} and \eqref{E:b_function_local} by $f =
g(x_{p_1},\ldots,x_{p_n})$ and $b = x_p h(x_{p_1},\ldots,x_{p_m})$ for smooth
functions $g$ and $h  > 0$, then
\begin{align*}
	v([f]) &= \sum_{i=1}^n \alpha(p_i) x_{p_i}(x)\pa_{x_{p_i}} g\big(x_{p_1}(x),\ldots,x_{p_n}(x)\big),
	\\ v'([b]) &= \alpha(p) + \sum_{i=1}^m \alpha(p_i) x_{p_i}(x)\pa_{x_{p_i}} \log h\big(x_{p_1}(x),\ldots,x_{p_m}(x)\big).
\end{align*}
It is an instructive exercise to check that \eqref{E:bT_local} is consistent
with relations between elements $p_i \in P$.

\begin{proof}
It is straightforward to check that $(v,v') \in \bT_x X_P$, and linearity of
\eqref{E:bT_local} is clear. The inverse map is given by $v' \mapsto \alpha$,
where $\alpha \in \Hom(P; \bbR)$ is determined by $\alpha(p) = v'([x_p])$.

That $\alpha$ is recovered from $(v,v')$ is clear. To see that $(v,v')$ is
recovered from $\alpha$, suppose that $(v,v') \in \bT_x X_P$. Since $v$ is a
derivation means that on a smooth function $f = g(x_{p_1},\ldots,x_{p_n})$ it
must act by
\[
	v([f]) = \sum_{i=1}^n c(x,p_i)\pa_{x_{p_i}} g\big(x_{p_1}(x),\ldots,x_{p_n}(x)\big),
\]
for some $c(x,p_i) \in\bbR$. To determine these coefficients, consider
$x_{p_i}$ as a smooth function and use \eqref{E:bT_condition} to deduce
\[
	c(x,p_i) = v([x_{p_i}]) = x_{p_i}(x)\,v'([x_{p_i}]) = x_{p_i}(x)\,\alpha(p_i).
\]

The final statement follows from the fact that $\Hom(P; \bbR) =
\Hom(P^\gp; \bbR)$ and $\dim(P) = \rank(P^\gp)$.
\end{proof}

The b-tangent spaces form the fibers of the {\bf b-tangent bundle} $\bT M \to
M$. From the local characterization in Proposition~\ref{P:bT_local}, this
inherits the structure of a smooth vector bundle of rank $\dim(M)$ in the
category of manifolds with generalized corners, which is canonically
trivialized over charts.
A b-map $f : M\to N$ induces a vector bundle morphism
\[
	\bd f_\ast : \bT M \to \bT N
\]
via the pull-back action on the sheaves $\C$ and $\B$.

\subsubsection{Normal bundles} \label{S:normal}
The inclusion $\iota_F : F \to M$ of a boundary face $F \in \cF(M)$ is a
non-interior b-map. Thus, for every $x \in F$,
\[
	\iota_F^\ast : \B_{M,x} \to \bB_{F,x} = \B_{F,x} \sqcup 0
\]
partitions $\B_x = \B_{M,x}$ into two submonoids, the {\bf normal} and {\bf
tangential} elements with respect to $F$:
\[
	\B_{x} = \cN(F,x) \sqcup \cT(F,x),
	\quad \cN(F,x) = (\iota_F^\ast)^{-1}(0),
	\quad \cT(F,x) =  (\iota_F^\ast)^{-1}(\B_{F,x}).
\]
In other words, $[b] \in \B_{x}$ is normal to $F$ if and only if
there is a neighborhood $U$ of $x$ in $F$ such that $b \rst U = 0$, and
tangential otherwise.
In particular, if $x \in F^\circ$, then $\cT(F,x) = \B_x^\times$ is just the
submonoid of units, as follows from Corollary~\ref{C:normal_nbhd}.

The {\bf b-normal space} to $F$ at $x$ is the real vector space
\[
	\bN_x F = \Hom(\B_{x}/\cT(F,x); \bbR)
\]
consisting of monoid homomorphisms from $\B_{x}$ to $\bbR$ which are trivial on
tangential elements to $F$.
The {\bf normal monoid} to $F$ at $x$ is
\[
	\bM_x F = \Hom(\B_{x}/\cT(F,x); \bbN) = (\B_{x}/\cT(F,x))^\vee
\]
and evidently $\bN_x F = \bM_x F \otimes_\bbN \bbR$.

There is a well-defined, injective linear map
\begin{equation}
	\bN_x F \to \bT_x M,
	\quad v' \mapsto (0,v').
	\label{E:bN_to_bT}
\end{equation}
Indeed, $v'$ determines a unique derivation $v \in \Der(\C_{x}; \bbR)$
satisfying \eqref{E:v_from_vprime}; however, since the image $\exp(\C_{x})
\subset \B_{x}^\times \subset \B_{x}$ is entirely tangential to $F$, we
must have $v = 0$.

For interior points $x \in F^\circ$, the normal monoid is $\bM_x F =
(\B_x/\cT(F,x))^\vee = (\B_x^\sharp)^\vee$, the dual to the conormal monoid
$W(F) = \B_x^\sharp$.
%
Using $\cT(F,x)$ rather than $\B_x^\times$ has the effect of extending this as
a bundle over the non-interior points of $F$ as well, as the next result shows.

\begin{prop}
For all $x \in F$, there is a natural monoid isomorphism
\begin{equation}
	\B_x/\cT(F,x) \cong W(F)
	\label{E:B_quotient}
\end{equation}
In particular, $\bN_x F \cong \Hom(W(F); \bbR)$ and $\bM_x F\cong
W(F)^\vee$.
\label{P:bN_local}
\end{prop}

\begin{proof}
This is a local statement, so it suffices to assume $M = X_P$ and $F = X_S$ for
some $S \leq P$.
Consider a b-function of the form $b = h\,x_p$, $h > 0$ on a neighborhood of
$x$. By positivity, $[h] \in \cT(X_S,x)$ so $[b] \equiv [x_p] \mod \cT(X_S,x)$.
From \eqref{E:boundary_inclusion} it follows that $x_p$ is tangential to $X_S$
if and only if $p \in S$ (while $x_p$ may vanish at points in the boundary of
$X_S$, it cannot vanish identically on a neighborhood unless $p \in P\setminus
S$). It follows that $P \to \B_{X_P,x}/\cT(X_S,x)$, $p \mapsto [x_p]$ is a
surjective homomorphism with kernel $S$, so that $\B_{X_P,x}/\cT(X_S,x) \cong
P/S$, which is the conormal monoid for $X_S$.
\end{proof}

From the local characterizations of $\bT_x M$ and $\bN_x F$ of
Propositions~\ref{P:bT_local} and \ref{P:bN_local}, respectively, it follows
that the map \eqref{E:bN_to_bT} extends to a short exact sequence
\[
	0 \to \bN_x F \to \bT_x M \to \bT_x F \to 0
\]
coinciding locally (i.e., on charts) with the sequence
\[
	0 \to \Hom(P/S; \bbR) \to \Hom(P; \bbR) \to \Hom(S; \bbR) \to 0
\]
for $M = X_P$, $F = X_S$.

The b-normal spaces $\bN_x F$ (resp.\ monoids $\bM_x F$) form the fibers of the
{\bf b-normal bundle} $\bN F \to F$ (resp.\ {\bf b-normal monoid bundle} $\bM F
\to F$) which inherits from Proposition~\ref{P:bN_local} the structure of a
smooth vector bundle of rank $\codim(F)$ (resp.\ bundle of monoids of dimension
$\codim(F)$) on $F$.

\begin{prop}
Let $f : M \to N$ be an interior b-map. Then for all $F \in \cF(M)$, the
differential $\bd f_\ast : \bT M \to \bT N$ restricts to a vector bundle
morphism
\begin{equation}
	\bd f_\ast : \bN F \to \bN G,
	\quad G = \cF(f)(F),
	\label{E:bdf_N}
\end{equation}
where $\cF(f)$ was defined in \eqref{E:f_sharp}, which further restricts to a
morphism of monoid bundles
\begin{equation}
	\bd f_\ast : \bM F \to \bM G,
	\label{E:bdf_M}
\end{equation}
with the property that each $\bd f_\ast : \bM_x F \to \bM_{f(x)} G$ is an
interior homomorphism of monoids.
\label{P:differential_normal}
\end{prop}
\begin{proof}
Let $x \in F$. Then
\begin{equation}
	f^\ast : \B_{N,f(x)} \to \B_{M,x}
	\label{E:f_ast_on_B}
\end{equation}
and we claim that $f^\ast$ maps $\cT(G,f(x))$ into $\cT(F,x)$ and $\cN(G,f(x))$
into $\cN(F,x)$.
By contradiction,
suppose that $[b] \in \cT(G,f(x))$ has image in $\cN(F,x)$. This means that $b
\rst G$ is a nontrivial (locally defined) b-function, but that $(f^\ast b) \rst
F \equiv 0$, which contradicts the defining property of $G = \cF(f)(F)$,
namely that $f : F \to G$ is an interior b-map.
Likewise, suppose that $[b] \in \cN(G,f(x))$ has image in $\cN(F,x)$. This means
that $b \rst G = 0$ but that $(f^\ast b) \rst F$ is nontrivial, which contradicts the
fact that $f(F) \subset G$.

It follows that $f^\ast$ descends to a monoid homomorphism 
\[
	f^\ast : \B_{N,f(x)}/\cT(G,f(x))  \to \B_{M,x}/\cT(F,x)
\]
with no nontrivial face in its kernel. The latter property is equivalent 
to the property that the dual homomorphism, $\bd f_\ast : \bM_x F \to \bM_{f(x)} G$, 
is an interior homomorphism.
\end{proof}

If $G$ and $F$ are two boundary faces of $M$ with $G \leq F$, then for all $x
\in G$ we have a natural homomorphism
\begin{equation}
	\bM_x F \to \bM_x G
	\label{E:mon_inclusion}
\end{equation}
coming from the fact that $\cT(G,x) \subset \cT(F,x) \subset \B_{M,x}$.

\begin{prop}
For all $x \in G \leq F$, the homomorphism \eqref{E:mon_inclusion} is an isomorphism
onto a boundary face of $\bM_x G$.

As a consequence, the bundles $\bM F \to F$ and $\bN F \to F$ are canonically
trivial for all $F \in \cF(M)$.
\label{P:pre_mon_cmplx}
\end{prop}
\begin{proof}
In any chart $O \subset X_P$ at $x$, $G$ and $F$ coincide with $X_T$ and
$X_S$ for some $T \leq S \leq P$, and then \eqref{E:mon_inclusion} is
identified via Proposition~\ref{P:bN_local} with the inclusion
\[
	(P/S)^\vee \cong S^\perp \leq T^\perp \cong (P/T)^\vee
\]
of faces of $P^\vee$.

For the second statement, first note that if $\codim(H) = 1$, then $\bM_x H =
\bbN$ is the unique toric monoid of dimension 1, and since $\Aut(\bbN) = 0$,
$\bM H \cong H\times \bbN$ has a unique trivialization.
For a general $F$, the inclusions $\bM_x H \to \bM_x F$ for all $H
\in \cF_1(M)$ with $F \leq H$ determine a labeling of the dimension 1 faces of
each $\bM_x F$, and the transition maps for the bundle $\bM F$ must be
consistent with these labelings. From Proposition~\ref{P:monoid_as_cone}, it is
easy to see that any automorphism of a toric monoid fixing the faces of
dimension 1 is trivial, so the local trivializations of $\bM F$ patch
together to form a global trivialization.
\end{proof}

It follows from this result that $\bM F$ may be identified with the trivial
monoid bundle
\[
	\bM F \cong F\times W(F)^\vee \to F,
\]
and for $G \leq F$, the map \eqref{E:mon_inclusion}, which must be independent
of $x \in G$, identifies $W(F)^\vee$ with a face of $W(G)^\vee$.
Putting this together with Proposition~\ref{P:differential_normal}, for a map
$f : M \to N$, the morphism \eqref{E:bdf_M}, which is locally constant on the
fibers, reduces to
\[
	\bd f_\ast \cong f \times f_\natural : F\times W(F)^\vee
	\to H \times W(H)^\vee,
	\quad H = \cF(f)(F)
\]
where $f_{\natural,F} : W(F)^\vee \to W(H)^\vee$ is a fixed homomorphism.
This is the principal motivation for the notion of monoidal complexes, which we
discuss next.

\subsection{Monoidal complexes} \label{S:mon_cplx}
A {\bf monoidal precomplex} is a category $\cP$ whose objects are monoids and
whose arrows, which we denote by
\begin{equation}
	Q \hto P
	\iff Q \stackrel \cong \to S \leq P,
	\label{E:mon_complex_arrow}
\end{equation}
are (injective) homomorphisms which factor as an isomorphism and the inclusion
of a face.
A precomplex is a {\bf monoidal complex} if for each object $P \in \cP$ and
face $S \leq P$, there exists a unique object $Q \in \cP$ and arrow $Q \hto P$
which is an isomorphism onto $S$.
It is sometimes convenient to identify $S$ and $Q$, writing simply $Q \leq P$,
but one must be careful since there may be multiple arrows $Q \hto P$ mapping
onto distinct faces.
A {\bf subcomplex} of a monoidal complex $\cP$ is a subcategory $\cS \subset
\cP$ which is also a monoidal complex.
If all monoids in a complex $\cP$ are sharp, then there is a unique initial
object $0 \in \cP$.

A {\bf morphism of monoidal complexes}, denoted
\[
	\phi : \cP \to \cQ,
\]
is, for each $P \in \cP$, an interior homomorphism $\phi_P : P \to Q$ (recall
that this means $\phi_P(P^\circ) \subset Q^\circ$) for some $Q \in \cQ$, such
that all the diagrams
\[
\begin{tikzcd}[column sep=small, row sep=small]
P \ar{r}{\phi_P} & Q \\ 
S \ar{r}{\phi_S} \ar[rightharpoonup]{u} & T \ar[rightharpoonup]{u}
\end{tikzcd}
\]
commute, where the vertical maps are arrows in $\cP$ and $\cQ$, respectively.
If $\cS$ is a subcomplex of $\cQ$, there is a canonical morphism $\cS \to \cQ$
with each homomorphism given by the identity.
If $\cS \subset\cQ$ is a subcomplex and $\phi : \cP \to \cS$ a morphism, then
\[
	\cP \rst \cS = \set{P \in \cP : \phi_P : P \to S,\ \text{for some}\ S \in \cS} \subset \cP
\]
is a subcomplex of $\cP$ and $\phi : \cP \rst \cS \to \cS$ is a morphism.

\begin{ex}
Fix a monoid $P$. The basic example of a monoidal complex is the set
\[
	\cP_P = \set{S : S \leq P}
\]
of faces of $P$ with arrows given by the inclusion homomorphisms.
Any subset $\cS \subset \cP_P$ for which $T \leq S$, $S \in \cS$ implies $T \in
\cS$ is a subcomplex.

A homomorphism $f : P \to Q$ determines a unique morphism $\phi_f : \cP_P \to
\cP_Q$ of monoidal complexes by the requirement that the $\phi_S$ be interior,
and vice versa.
\label{X:monoid_as_complex}
\end{ex}

It is convenient to identify the monoidal complex $\cP_P$ with $P$ itself,
which we shall do from now on when no confusion can arise.

\begin{rmk}
A monoidal complex is closely related to Kato's notion of a `fan' \cite{kato}. See
the discussion in \S\ref{S:commentary}.
\end{rmk}

We may summarize the observations at the end of the previous section in the
following Theorem.

\begin{thm}
For every connected, finite dimensional manifold with generalized corners $M$, there
is a monoidal complex
\[
	\cP_M = \set{W(F)^\vee : F \in \cF(M),\ \bM F \cong F\times W(F)^\vee}
\]
indexed by boundary faces of $M$, and every interior b-map $f : M \to N$ induces
a morphism
\[
	f_\natural : \cP_M \to \cP_N.
\]
The association $M \mapsto \cP_M$, $f \mapsto f_\natural$ is a functor from the
category of manifolds with generalized corners and interior b-maps to the
category of monoidal complexes.
\label{T:P_M}
\end{thm}

It is worthwhile to spell this out more explicitly for the model space $X_P$,
as we will make heavy use of this below.

\subsubsection{Monoidal complexes for model spaces} \label{S:mon_cmplx_models}
Let $P$ be a monoid, not necessarily sharp.
The boundary faces of $X_P$ are the model spaces $X_S$ for $S \leq P$, and we
have
\[
	W(X_S)^\vee = (P/S)^\vee \cong S^\perp \leq P^\vee
\]
For $T\leq S$, the inclusion $X_T \leq X_S$ of boundary faces induces the
arrow $W(X_S)^\vee \hto W(X_T)^\vee$ which is associated with the inclusion
\[
	W(X_S)^\vee \cong S^\perp \leq T^\perp \cong W(X_T)^\vee
\]
of faces of $P^\vee$.
In particular, $W(X_P)^\vee = \set 0$ and $W(X_{P^\times})^\vee = P^\vee$.
It follows that
\[
	\cP_{X_P} \cong P^\vee
\]
is identified with the complex of faces $\set{S^\perp : S \leq P}$ of the dual
monoid $P^\vee$.

For an open set $O \subset X_P$,
\[
	\cP_O = \set{S^\perp : O \cap X_S \neq \emptyset}
\]
is the subcomplex of $P^\vee$ consisting of monoids $S^\perp$ for which $O$
meets the boundary face $X_S$.
By making $O$ smaller if necessary, it is often convenient to assume that $O$
meets a unique minimal boundary face $X_S$; in particular $O \subset
X_{S^{-1}P}$. Then by replacing $P$ by $S^{-1}P$ if necessary, we may assume
that $O$ meets the minimal boundary face $X_{P^\gp}$, and then
\[
	\cP_O = \cP_{X_P} = P^\vee.
\]

If $f : O \subset X_P \to X_Q$ is a map, then $f_\natural : \cP_O
\to \cP_{X_Q}$ may be identified with a single homomorphism
\begin{equation}
	f_\natural : P^\vee \to Q^\vee.
	\label{E:f_natural_as_homomorphism}
\end{equation}
The image of this homomorphism is contained in the face $T^\perp \leq Q^\vee$,
where $X_T = \cF(f)(O \cap X_{P^\times})$ is the smallest boundary face
containing the image of the minimal boundary face of $O$.
In particular, the image, $f(O)$, of $O$ is contained in the open set
$X_{T^{-1}Q} \subset X_Q$, so by replacing $Q$ by $T^{-1} Q$ if necessary, we
can assume that $T = Q^\times$, and then \eqref{E:f_natural_as_homomorphism} is
an interior homomorphism.

Finally, we are in a position to relate this to the local coordinate form for
$f$.
\begin{lem}
Let $f : O \subset X_P \to X_Q$ be a map with $O \cap X_{P^\times} \neq
\emptyset$ and $\cF(f)(O \cap X_{P^\times}) = X_{Q^\times}$, and choose
generators $\set{p_i,\pm q_j : 1 \leq i \leq n,\,1 \leq j \leq m}$ for
$P$ and $\set{p'_i,\pm q'_j : 1 \leq i \leq n',\,1 \leq j \leq m'}$ for
$Q$ as in \eqref{E:gens_nonsharp}, with associated coordinates $(x,y)$ on $X_P$
and $(x',y')$ on $Q$, with respect to which $f$ is given in coordinates as in
\eqref{E:local_b_map} by
\[
	f(x,y) = \big(h(x,y)\,x^\mu, g(x,y)\big) = (x',y').
\]
Then $\mu \in \Mat(n\times n',\bbN)$ is the matrix representing the dual
homomorphism to \eqref{E:f_natural_as_homomorphism} with respect to the
generators $\set{p_1,\ldots,p_n}$ of $P^\sharp$ and $\set{p'_1,\ldots,p'_{n'}}$
of $Q^\sharp$:
\[
	f_\natural^\vee : Q^\sharp \to P^\sharp,
	\quad f_\natural^\vee(p'_j) = \sum_{i=1}^n \mu_{ij} p_i.
\]
\label{L:exponents}
\end{lem}
\begin{proof}
This is a matter of unwinding the definitions.
The homomorphism \eqref{E:f_natural_as_homomorphism} is determined by
\[
	\bd f_\ast : \bM_{(0,y)} X_{P^\times} \cong P^\vee \to Q^\vee \cong \bM_{f(0,y)} X_{Q^\times}
\]
for any point $(0,y) \in X_{P^\times}$ in the minimal boundary face.
Then
we may identify $\bM_{(0,y)} X_{P^\times} \cong (P^\sharp)^\vee$ with the dual stalk
$(\B_{(0,y)}^\sharp)^\vee$ and likewise identify $\bM_{f(0,y)} X_{Q^\times} \cong
(Q^\sharp)^\vee$ with $(\B_{f(0,y)}^\sharp)^\vee$. The dual of the homomorphism
$\bd f_\ast$ is just the pull-back
\[
	f^\ast : \B_{f(y,0)}^\sharp \to \B_{(0,y)}^\sharp,
\]
sending $[x_{p_j'}]$ to $[h'_j\,\prod_i x_{p_i}^{\mu_{ij}}] = \prod_i [x_{p_i}]^{\mu_{ij}}$,
and we recall that the isomorphism $\B_{(0,y)}^\sharp \cong P^\sharp$ is given by
identifying $[x_p]$ with $p$.
\end{proof}

\subsubsection{Refinement} \label{S:refinement}

A {\bf saturated refinement} is a morphism $\psi :\cR \to \cP$ of monoidal
complexes with the following properties:
\begin{enumerate}
[{\normalfont (R1)}]
\item  \label{I:refinement_inj}
$\psi_R : R \to P$ is injective for all $R \in \cR$, and
\item \label{I:refinement_cover}
for each $P \in \cP$, there is a partition
\[
	P^\circ = \bigsqcup_{R \to P} \psi_R(R^\circ)
\]
of the interior of $P$ into the images of the interiors of all the $R \in \cR$
mapping to $P$.
\end{enumerate}
In other words, for any pair $R_1,R_2 \in \cR$ mapping to $P \in \cP$, their
images in $P$ may only intersect at a mutual face:
\[
	\psi_{R_1}(R_1) \cap \psi_{R_2}(R_2) = \psi_S(S) \subset P
	\ \text{for some}\  S \hto R_i,\ i = 1,2.
\]
A saturated refinement encodes the idea of consistently {\em subdividing} each
of the monoids $P$ of $\cP$ into toric submonoids meeting along mutual faces,
with the consistency condition that the induced subdivision of $S \leq P$
coincides with the subdivision of $Q$ in \eqref{E:mon_complex_arrow}.

In particular, in the case the $\cP = P$ is the complex associated to a single
monoid, a refinement $\cR \to P$ may be identified with a collection
$\set{R_i\subset P}$ of submonoids with $\dim(R_i) = \dim(P)$ subject to the
conditions that $P = \bigcup_i R_i$ and $R_1 \cap R_2 = S$ such that $S \leq
R_i$, $i = 1,2$.

The following is immediate.

\begin{prop}
Let $\cR \to \cP$ be a refinement and $\cS \subset \cP$ a subcomplex. Then $\cR
\rst \cS \to \cS$ is a refinement.
\label{P:restricting_refinements}
\end{prop}

\begin{rmk}
There is a weaker notion of not necessarily saturated refinement (see
\cite{KM}), wherein the condition (R\ref{I:refinement_cover}) is replaced by an
analogous condition on the polyhedral cones supporting $P$ as in
Proposition~\ref{P:monoid_as_cone}, but where the images $\psi_R(R)$ need not
be saturated submonoids in $P$. For example, the map $\bbN \to \bbN$, $n
\mapsto 2n$ is a refinement which is not saturated.

It may be possible to generalize the notion of blow-up developed below to
non-saturated refinements, as was done for smooth refinements in
\cite{KM}. However, the technical machinery needed to implement this is quite
non-algebraic, and will not be considered here.
\end{rmk}

\medskip\noindent
From this point on,
{\em refinement} will mean {\em saturated
refinement}.
\medskip

\section{Blow-up} \label{S:blow-up}
\subsection{Blow-up of model spaces} \label{S:blow-up_models}
Fix a sharp monoid $P$ and a refinement $\psi : \cR \to P^\vee \cong
\cP_{X_P}$.
Say $R \in \cR$ is {\em maximal} if $\dim(R) = \dim(P^\vee)$.

For each maximal $R$, the dual homomorphism $\psi_R^\vee : P \to R^\vee$
determines an interior b-map
\begin{equation}
	X_{R^\vee} \to X_P,
	\label{E:refinement_map_on_models}
\end{equation}
such that $X_{R^\vee}^\circ \to X_P^\circ$ is a diffeomorphism.
We proceed below to glue together the spaces $X_{R^\vee}$ together to form a
new manifold mapping to $X_P$.

We identify each $R \in \cR$ its image in $P^\vee$, thus
\[
	P^\vee = \bigcup_{R_i\text{ max'l}} R_i,
	\quad R_1 \cap R_2 = Q, \quad Q\leq R_i,\ i = 1,2.
\]

\begin{lem}
Let $R_1,R_2 \in \cR$ be maximal, with $Q = R_1 \cap R_2$.
Then there is a natural diffeomorphism
\begin{equation}
	X_{(Q^\perp)^{-1}R_1^\vee} \cong X_{(Q^\perp)^{-1}R_2^\vee}
	\label{E:local_diffeos}
\end{equation}
between the dense open sets $X_{(Q^\perp)^{-1}R_i^\vee} \subset X_{R_i^\vee}$, $i = 1,2$,
which is consistent with the maps \eqref{E:refinement_map_on_models}.
\label{L:gluing_maps}
\end{lem}

\begin{proof}
This follows from an isomorphism $(Q^\perp)^{-1} R_1^\vee \cong
(Q^\perp)^{-1}R_2^\vee$, which, having identified $R_i$ with their images in
the lattice $L := (P^\vee)^\gp$ containing $P^\vee$, takes the form of an
equality of monoids
\begin{equation}
	(Q^\perp)^{-1} R_1^\vee =  (Q^\perp)^{-1}R_2^\vee \subset L^\vee
	\label{E:gluing_equality}
\end{equation}
in the dual lattice $L^\vee = \Hom(L; \bbZ) \cong P^\gp$.
In this lattice we have
\begin{align*}
	R_i^\vee &= \set{r \in L^\vee : r(R_i) \subset \bbN},
	\\ Q_i^\perp &= \set{q \in L^\vee : q(R_i) \subset \bbN, \ q(Q) = 0},
	\\ (Q_i^\perp)^\gp = (Q^\perp)^\gp &= \set{q \in L^\vee : q(Q) = 0},
	\\ (Q^\perp)^{-1}R_i^\vee &= \set{r + q  \in L^\vee : r(R_i) \subset \bbN, q(Q) = 0}.
\end{align*}
(We have to distinguish between the faces $Q_i^\perp \leq R_i^\vee$ since
$Q_1^\perp \neq Q_2^\perp$; however their group completions are the same.) To
prove \eqref{E:gluing_equality}, it suffices to show that
$(Q^\perp)^{-1}R_1^\vee \subset (Q^\perp)^{-1}R_2^\vee$ by symmetry, so
consider an element $s$ in the first set.
This has the property that $s(Q) \subset \bbN$.
Let $\set{p_1,\ldots,p_n}$ be a finite set of generators for $R_2$, and suppose
that $\set{p_1,\ldots,p_k} \subset R_2 \setminus Q$ with
$\set{p_{k+1},\ldots,p_n} \subset Q$.
Define the integers $l_i = s(p_i) \in \bbZ$ for $i = 1,\ldots,k$.
For each $i$, if $l_i < 0$, choose $q_i \in Q^\perp$ such that $q_i(p_i) \geq
-l_i$ (such an element exists by Lemma~\ref{L:nontriv_dual}), otherwise set
$q_i = 0$.
Then
\begin{gather*}
	s = r + q \in (Q^\perp)^{-1}R_2^\vee, \ \text{where}
	\\ r = s + \textstyle \sum_i q_i \in R_2^\vee,
	\quad \text{and} \quad q = s - r = - \textstyle \sum_i q_i \in (Q^\perp)^\gp.
\end{gather*}

The consistency of \eqref{E:local_diffeos} with
\eqref{E:refinement_map_on_models} follows from the obvious commutativity of
the two homomorphisms $P \to (Q^\perp)^{-1} R_1^\vee = (Q^\perp)^{-1}R_2^\vee$
through $R_1^\vee$ and $R_2^\vee$.
\end{proof}

The {\bf blow-up} of $X_P$ by $\cR$ is the push-out of the $X_{R^\vee}$ along
the sets $X_{(Q^\perp)^{-1}R^\vee}$ for maximal $R \in \cR$:
\begin{equation}
	[X_P; \cR] = \Big(\textstyle\bigcup_{R\text{ max'l}} X_{R^\vee} \Big)/ \sim
	\label{E:local_blow-up}
\end{equation}
where the equivalence relation $\sim$ is generated by the diffeomorphisms
\eqref{E:local_diffeos}: $X_{R_1^\vee} \ni x_1 \sim x_2 \in X_{R_2^\vee}$ if
$x_i \in X_{(Q^\perp)^{-1}R_i^\vee}$ and they are identified by such a
diffeomorphism.
The {\bf blow-down map}
\begin{equation}
	\beta : [X_P; \cR] \to X_P
	\label{E:local_blow_down}
\end{equation}
is well-defined by \eqref{E:refinement_map_on_models} and
Lemma~\ref{L:gluing_maps}.

\begin{prop}
The blow-up $[X_P; \cR]$ is a manifold whose monoidal complex is isomorphic to
$\cR$, and the blow-down map \eqref{E:local_blow_down} is an interior b-map
with $\beta_\natural$ coinciding with the refinement morphism $\psi : \cR \to P^\vee$.
Moreover, $\beta$ is a diffeomorphism from $[X_P; \cR]^\circ$ to $X_P^\circ$.
\label{P:local_blow-up_strucure}
\end{prop}
\begin{proof}
The space $[X_P; \cR]$ has an open cover by the charts $X_{R^\vee}$ with
diffeomorphic transition maps. 
To see that $[X_P; \cR]$ is Hausdorff, it suffices to show that two points $x_i
\in X_{R_i^\vee} \setminus X_{(Q^\perp)^{-1}R_i^\vee}$, $i = 1,2$ are separated
in the quotient \eqref{E:local_blow_down}. Note that $x_i \notin
X_{(Q^\perp)^{-1}R_i^\vee}$ means that $x_i(q) = 0$ for $q \in Q_i^\perp \leq
R_i^\vee$.

Since $R_1$ and $R_2$ only intersect along the mutual face $Q$, there is an
element $q \in (Q^\perp)^\gp \subset L^\vee$ such that $q(R_1) \subset \bbN$,
$q(R_2) \subset - \bbN$, and $q(Q) = 0$. In other words, $q \in Q_1^\perp$ and
$-q \in Q_2^\perp$.
Then 
\[
	x_1 \in U_1 = \set{x_q < \varepsilon} \subset X_{R_1^\vee},
	\quad x_2 \in U_2 = \set{x_{-q} < \varepsilon} \subset X_{R_2^\vee},
\]
since $x_q(x_1) = x_{-q}(x_2) = 0$.
However, in $X_{(Q^\perp)^{-1}R_2^\vee} \cong X_{(Q^\perp)^{-1}R_1^\vee}$, the
set $U_2$ is identified with the set $\set{x_q > \varepsilon^{-1}}$, so $U_1
\cap U_2 = \emptyset$ for $\varepsilon < 1$.

The boundary faces of $X_{R^\vee}$ are the subspaces $X_{T^\perp} \subset
X_{R^\vee}$ for all $T \leq R$, which is to say that $\cP_{X_R^\vee} = \set{ T
: T \leq R} \subset \cR$.
However, if $T$ is a mutual face of both $R_1$ and $R_2$, then the interiors
$X_{T^\perp}^\circ \subset X_{R_i^\vee}$ are identified by
\eqref{E:local_diffeos}. In particular there is a bijection between $T \in \cR$
and faces of $[X_P; \cR]$, and it follows that $\cP_{[X_P; \cR]} \cong \cR$.

That $[X_P; \cR]^\circ \cong X_P^\circ$ follows from $X_{R^\vee}^\circ \cong
X_P^\circ$ and the fact that $X_{R^\vee}^\circ \subset
X_{(Q^\perp)^{-1}R^\vee}$ for each $Q \leq R$.
\end{proof}

\subsubsection{Blow-up and localization} \label{S:blow-up_localization}

Next we investigate how the blow-up behaves with respect to passing to the open
subsets $X_{S^{-1}P} \subset X_P$ for $S \leq P$.

From $(P/S)^\vee \cong S^\perp \leq P^\vee$, we may regard $(P/S)^\vee \subset
P^\vee$ as a monoidal subcomplex. The restriction of the refinement $\cR \to
P^\vee$ to the subcomplex $(P/S)^\vee$ is again a refinement, so defines a
blow-up of the model space $X_{P/S}$.

\begin{prop}
Then the preimage of the open set $X_{S^{-1}P} \cong X_{P/S}\times X^\circ_{S}
\subset X_P$ under \eqref{E:local_blow_down} admits a diffeomorphism
\begin{equation}
	\beta^{-1}(X_{S^{-1}P}) \cong [X_{P/S}; \cR \rst (P/S)^\vee]\times X^\circ_{S}.
	\label{E:blow-up_localization}
\end{equation}
\label{P:blow-up_localization}
\end{prop}

Proposition~\ref{P:blow-up_localization} suggests a way to define the blow-up
of $X_P$ for a non-sharp monoid $P$.
Namely, if $P$ is not sharp, we define $[X_P; \cR]$ for a refinement $\cR \to
P^\vee = (P^\sharp)^\vee$ by
\[
	[X_P; \cR] = [X_{P^\sharp}; \cR] \times X_{P^\times}
\]
with respect to an isomorphism $P \cong P^\sharp \times P^\times$, and then
this is consistent with further localization.

\begin{lem}
For any $R \in \cR$, let $T = R \cap S^\perp \leq R$ be the face given by the
intersection of $R$ with the face $S^\perp \leq P^\vee$.  Then the preimage of
$X_{S^{-1}P}$ under \eqref{E:refinement_map_on_models} is the space
$X_{(T^\perp)^{-1}R^\vee}$, and we have a commutative diagram
\[
\begin{tikzcd}[column sep=small, row sep=small]
X_{(T^\perp)^{-1}R^\vee} \ar[hookrightarrow]{r} \ar{d}&  X_{R^\vee} \ar{d} \\
X_{S^{-1}P} \ar[hookrightarrow]{r} &  X_P 
\end{tikzcd}
\]
\label{L:refinement_subspace_factor}
\end{lem}

\begin{proof}
The preimage of $X_{S^{-1}P} = \Hom(S^{-1}P; \bbR_+)$  in $X_{R^\vee}$ consists
of those $x \in \Hom(R^\vee; \bbR_+)$ such that $x \neq 0$ on the image of $S$
in $R^\vee$ with respect to the homomorphism $P \to R^\vee$.
Since $T \subset S^\perp$, by duality $S \subset T^\perp$, and $S$ does not lie
in any proper boundary face of $T^\perp$ since then there would be a larger $Q
\leq R$ for which $Q \subset S^\perp$.
Since $x^{-1}(0)$ is the complement of a face of $R^\vee$, and $x \neq 0$ on
$S$, it follows that $x \neq 0$ on $T^\perp$ and thus $x \in
X_{(T^\perp)^{-1}R^\vee}$.
\end{proof}

\begin{proof}[Proof of Proposition~\ref{P:blow-up_localization}]
It follows from Lemma~\ref{L:refinement_subspace_factor}, that
$\beta^{-1}(X_{S^{-1}P})$ consists of the union as in \eqref{E:local_blow-up}
of the subspaces $X_{T^\perp)^{-1}R^\vee} \subset X_{R^\vee}$ for maximal $R
\in \cR$, where $T = R \cap S^\perp \leq R$.
On the other hand, if we denote $\cR \rst (P/S)^\vee$ by $\cT$ (viewed as
a set of submonoids of $S^\perp$), the blow-up $[X_{P/S}; \cT]$ is determined
by the gluing of the spaces $X_{T^\vee}$ for the set of $T \in \cT$ which are
maximal, i.e., $\dim(T) = \dim(P/S) = \codim(S)$.

It is easy to see that each maximal $T \in \cT$ is a face $T \leq R$ for some
maximal $R \in \cR$, but the converse is false; for $R \in \cR$ maximal, the
corresponding $T \in \cT$ given by $T = R \cap S^\perp$ need not be maximal.
However, in this latter situation there necessarily exists some other maximal
$R' \in \cR$ with a maximal face $T' \in \cT$, for which $T \leq T'$, and $T
\leq Q$, where
\[
	Q = R \cap R' \subset P^\vee,
\]
as in Lemma~\ref{L:gluing_maps}.
It follows that $(Q^\perp)^{-1}R^\vee \subset (T^\perp)^{-1}R^\vee$ and
therefore that $X_{(T^\perp)^{-1} R^\vee}$ is entirely contained in the subset
$X_{(Q^\perp)^{-1}R^\vee}$ along which $X_{R^\vee}$ and $X_{(R')^\vee}$ are
glued.

Thus it suffices to restrict consideration to those maximal $R \in \cR$ with a
maximal $T \leq R$, $T \in \cT$.
For such $R$ and $T$, we claim that $X_{(T^\perp)^{-1} R^\vee} \cong
X_{T^\vee}\times X^\circ_S$.

Dualizing the maps $T \hookrightarrow R$, $S^\perp \hookrightarrow P^\vee$ and
$R \hookrightarrow P^\vee$, we have a commutative diagram
\[
\begin{tikzcd}[column sep=small, row sep=small]
T^\perp \ar[hookrightarrow]{r} \ar{d}& R^\vee \ar{d}
\\ S \ar[hookrightarrow]{r} & P
\end{tikzcd}
\]
Passing to the localizations of $S$ and $T^\perp$, and using the assumption
that $\dim(T) = \dim(S^\perp)$, so that $\dim(T^\perp) = \dim(S)$, we have
\[
\begin{tikzcd}[column sep=small, row sep=small]
(T^\perp)^\gp \ar[hookrightarrow]{r} \ar[-, double equal sign distance]{d}& (T^\perp)^{-1}R^\vee  \ar{d}
\\ S^\gp \ar[hookrightarrow]{r} & S^{-1}P
\end{tikzcd}
\]
in which the left vertical arrow is an equality.

The isomorphism $S^{-1}P \cong (P/S) \times S^\gp$ comes from a choice of
splitting of the exact sequence
\[
	S^\gp \to P^\gp \to P^\gp/S^\gp
\]
of abelian groups, and by the above this determines a compatible isomorphism
\[
	(T^\perp)^{-1}R^\vee \cong T^\vee \times (T^\perp)^\gp = T^\vee \times S^\gp
\]
with respect to which the map $X_{(T^\perp)^{-1}R^\vee} \to X_{S^{-1}P}$ is
identified with the map $X_{T^\vee}\times X_S^\circ \to X_{P/S}\times
X_S^\circ$, with the identity map in the second factor.
\end{proof}

\subsubsection{Local blow-up and b-maps} \label{S:blow-up_and_b-maps}
Suppose $f : O \subset X_P \to X_Q$ is a b-map such that $O$ meets the minimal
boundary face $X_{P^\times}$ and $\cF(f)(O\cap X_{P^\times}) = X_{Q^\times}$,
so that $f_\natural : \cP_O \to \cP_{X_Q}$ is determined by a single
homomorphism $f_\natural : P^\vee \to Q^\vee$ as in \S\ref{S:mon_cmplx_models}.
\begin{lem}
Suppose $f : O \subset X_P \to X_Q$ is a map as above, and $\cR \to Q^\vee$
is a refinement.  If $f_\natural : P^\vee \to Q^\vee$ factors through some $R
\in \cR$, then there is a unique lift of $f$ to a map
\[
	\wt f : O \subset X_P \to [X_Q; \cR]
\]
such that $f = \beta \circ \wt f$.
\label{L:local_b-map_lift}
\end{lem}
\begin{proof}
We will show that $f$ lifts to factor uniquely through the map $X_{R^\vee} \to
X_Q$.
The existence of such a map is obtained using coordinates. Thus let
$\set{p_1,\ldots,p_n}$, $\set{p'_1,\ldots,p'_{n'}}$ and $\set{r_1,\ldots,r_l}$
be generators for $P^\sharp$, $Q^\sharp$ and $R^\vee = R$, respectively.
By assumption we have a commutative diagrams
\[
\begin{tikzcd}[column sep=small]
Q^\sharp \ar{rr}{f_\natural^\vee} \ar{dr}[below left]{\beta_\natural^\vee} & & P^\sharp 
\\ & R^\vee \ar{ur}[below right]{\psi} &
\end{tikzcd}
\qquad
\begin{tikzcd}[column sep=tiny]
(Q^\sharp)^\gp \ar{rr}{(f_\natural^\vee)^\gp} \ar{dr}[above right]{\cong}[below left]{(\beta_\natural^\vee)^\gp}
	&& (P^\sharp)^\gp 
\\ & (R^\vee)^\gp \ar{ur}[below right]{\psi^\gp}& 
\end{tikzcd}
\qquad
\begin{tikzcd}[column sep=small]
\bbZ^{n'} \ar{rr}{\mu} \ar{dr}[below left]{\nu} & & \bbZ^n 
\\ & \bbZ^l \ar{ur}[below right]{\gamma} & 
\end{tikzcd}
\]
where each diagram is a restriction of the latter ones. Here
$\mu \in \Mat(n\times n'; \bbN)$ represents $f_\natural^\vee$, $\nu \in
\Mat(n\times l; \bbN)$ represents $\beta_\natural^\vee$ and $\gamma \in
\Mat(l\times n'; \bbN)$ represents $\psi$ with respect to the chosen generators, and we
realize $(P^\sharp)^\gp$, $(Q^\sharp)^\gp$ and $(R^\vee)^\gp$ as sublattices in $\bbZ^n$,
$\bbZ^{n'}$ and $\bbZ^l$, respectively.

From Lemma~\ref{L:exponents} and the definition of $X_{R^\vee} \to X_P$, the
maps of spaces are represented with respect to coordinates $(x,y)$ on $X_P$,
$(x',y')$ on $X_Q$ and $(x'',y'')$ on $X_{R^\vee}\times X_{Q^\times}$ by 
\begin{align*}
	f &: (x,y) \mapsto \big(h(x,y)\,x^\mu, g(x,y)\big) = (x',y')
	\\ \beta &: (x'',y'') \mapsto \big((x'')^\nu, y'') = (x',y').
\end{align*}
Since $(\beta_\natural^\vee)^\gp : (Q^\sharp)^\gp \to (R^\vee)^\gp$ is an isomorphism,
there exists a (not necessarily unique) $\tau : \bbZ^l \to \bbZ^{n'}$ such that
$\tau\nu = 1$ on the subspace $(Q^\sharp)^\gp \subset \bbZ^{n'}$.
Note that $\tau$ may have negative entries.
In particular, $\mu\tau = \gamma$ on $(R^\vee)^\gp$;
equivalently, $\tau^\vee\mu^\vee = \gamma^\vee$ on $(P^\vee)^\gp$.
Define the b-map $\wt f : O \to X_R$ in coordinates by
\[
	\wt f : (x,y) \mapsto \big((h(x,y)\,x^\mu)^\tau, g(x,y)\big).
	= \big(h(x,y)^\tau\,x^\gamma, g(x,y)\big)
\]
That $\tau$ may have negative entries is not a problem since the components of
$h$ are strictly positive.

Composing with $\beta$ gives
\begin{align*}
	\beta \circ \wt f : (x,y) \mapsto \big((h(x,y)^\tau\,x^\gamma)^\nu, g(x,y)\big)
	&= \big(((h(x,y)\,x^\mu)^\tau)^\nu, g(x,y)\big)
	\\&= \big(h(x,y)\,x^\mu, g(x,y)\big).
\end{align*}
This proves existence.

Uniqueness follows from the fact that $X_{R^\vee}^\circ \cong X_Q^\circ$, so
that $\wt f$ is completely determined by $f$ as a map from $O \cap X_P^\circ$
to $X_{R^\vee}^\circ$, and then then extension to the whole domain is unique by
continuity.
\end{proof}

\begin{cor}
Suppose $O_i \subset X_{P_i}$, $i = 1,2$ are open sets with a diffeomorphism
\[
	f : O_1 \stackrel \cong \to O_2,
\]
and suppose $\cR_i \to P_i^\vee$ are refinements such that $\cR_1 \rst
\cP_{O_1} \to \cP_{O_1} \stackrel{f_\natural} \to \cP_{O_2}$ factors through an
isomorphism
\[
	\cR_1 \rst \cP_{O_1} \cong \cR_2 \rst \cP_{O_2}.
\]
Then $f$ lifts to a unique diffeomorphism
\[
	\wt f : O'_1 \stackrel \cong \to O'_2,
	\quad O'_i = \beta^{-1}(O_i) \subset [X_{P_i}; \cR_i],\ i = 1,2,
\]
between the preimages of the $O_i$ in the blown up spaces.
\label{C:diffeos_lift}
\end{cor}
\begin{proof}
The composition of the blow-down and $f$ is a b-map
\[
	f\circ \beta_1 : O'_1 \to X_{P_2}
\]
whose morphism $\cP_{O'_1} \cong \cR_1 \rst \cP_{O_1}\to \cP_{X_{P_2}} =
P_2^\vee$ factors through $\cR_2 \to P_2^\vee$ by assumption. 
From Lemma~\ref{L:local_b-map_lift}, this lifts to a b-map $\wt f : O'_1 \to
[X_{P_2}; \cR_2]$ with image in $O'_2$. 
Likewise, the lift of $f^{-1}\circ \beta : O'_2 \to X_{P_1}$ is a map $\wt g :
O'_2 \to [X_{P_1}; \cR_1]$ with image in $O'_1$.

Observe that $f \circ \beta_1 \circ \wt g = \beta_2$ and $f^{-1}\circ \beta_2
\circ \wt f = \beta_1$, and since the respective identity maps on $O'_1$ and
$O'_2$ are b-maps lifting $\beta_1$ and $\beta_2$, it follows from the
uniqueness part of Lemma~\ref{L:local_b-map_lift} that $\wt f$ and $\wt g$ are
inverses.
\end{proof}

\subsection{Global blow-up} \label{S:global_blow-up}

Let $\set{O'_a}_{a \in A}$ be an atlas for a manifold $M$, which is to say a
cover by charts $\phi_a : O_a' \stackrel \cong\to  O_a \subset X_{P(a)}$; by
shrinking the $O_a'$ if necessary, we assume that each $O_a'$ meets a unique
minimal boundary face $F_a \in \cF(M)$. 
(In other words, among those $F \in \cF(M)$ for which $O_a' \cap F \neq
\emptyset$, there is a unique one with codimension $\max\set{\codim(F) : O_a'
\cap F \neq \emptyset}$.)
Replacing $P(a)$ by $S^{-1}P(a)$ if necessary, we may also assume that $F_a
\cap O_a'$ is identified with the minimal boundary face $X_{P(a)^\times}
\subset X_{P(a)}$; in particular $P(a)^\sharp \cong W(F_a)$. 

Then $M$ admits a canonical diffeomorphism
\begin{equation}
	M \cong \Big(\bigcup_{a \in A} O_a \Big)/ \sim
	\label{E:M_as_quotient}
\end{equation}
where $\sim$ is the equivalence relation generated by transition
diffeomorphisms
\begin{equation}
	O_{ba} \subset X_{P(a)}\stackrel \cong \to O_{ab} \subset X_{P(b)},
	\label{E:trans_diffeos_for_blowup}
\end{equation}
where $O_{ba} = \phi_a(O'_a \cap O_b') \subset O_a$ and $O_{ab} = \phi_b(O'_a
\cap O'_b) \subset O_b$.

Let $\cR \to \cP_M$ be a refinement.  For each $F \in \cF(M)$, we regard
$W(F)^\vee \subset \cP_M$ as a monoidal subcomplex, so the restriction $\cR
\rst W(F)^\vee$ may be identified with a refinement of $P(a)^\vee =
(P(a)^\sharp)^\vee$ whenever $F_a = F$. Let
\[
	U_a = \beta_a^{-1}(O_a) \subset [X_{P(a)}; \cR \rst W(F_a)^\vee].
\]
By Corollary~\ref{C:diffeos_lift}, the transition diffeomorphisms
\eqref{E:trans_diffeos_for_blowup} lift canonically to diffeomorphisms
\begin{equation}
	U_{ba} := \beta_a^{-1}(O_{ba}) \stackrel \cong \to U_{ab} = \beta_b^{-1}(O_{ab}),
	\label{E:lifted_diffeos}
\end{equation}
and we define the {\bf blow-up} of $M$ by
\begin{equation}
	[M; \cR] = \Big(\bigcup_{a \in A} U_a\Big)/ \sim,
	\quad \beta : [M; \cR] \to M,
	\label{E:global_blow-up}
\end{equation}
where $\sim$ is the equivalence relation generated by \eqref{E:lifted_diffeos}.
The {\bf blow-down map} $\beta$ in \eqref{E:global_blow-up} is determined by
the quotients of the blow-down maps $\beta_a : U_a \to O_a$.

\begin{thm}
For each refinement $\psi: \cR \to \cP_M$, $[M; \cR]$ is a manifold which is
well-defined up to unique diffeomorphism, with $[M; \cR]^\circ \cong M^\circ$,
$\cP_{[M;\cR]} \cong \cR$ and $\beta_\natural \cong \psi : \cR \to \cP_M$. The
blow-up has the following universal property: If $f : N \to M$ is an interior
b-map such that $f_\natural : \cP_N \to \cP_M$ factors through $\psi$, then $f$
lifts to a unique interior b-map $\wt f : N \to [M; \cR]$ such that $f = \beta
\circ \wt f$.
\label{T:main_thm}
\end{thm}
\begin{proof}
If two points $x_1 \neq x_2 \in [M; \cR]$ satisfy $\beta(x_1) = \beta(x_2)$,
then they lie in the Hausdorff subspace $U_a$ for some $a$, and if $\beta(x_1)
\neq \beta(x_2)$ then they are separated by the preimage under $\beta$ of
separating open sets in $M$, so $[M;\cR]$ is a Hausdorff space, with the
structure of a manifold with generalized corners coming from the charts on the
$U_a$.

In $M$, each $F^\circ$ has a covering by charts $O'_a$, such that (shrinking
$O'_a$ if necessary), $O'_a \cong O_a \subset X_{W(F)} \times X_{\bbZ^l}$, $l =
\dim(F)$, with the transition diffeomorphisms preserving this product
structure. 
The preimages of these charts in $[M; \cR]$ have the form $U_a \subset
[X_{W(F)}; \cR \rst W(F)^\vee]\times X_{\bbZ^l}$, with transition
diffeomorphisms again preserving the product structure. 
It follows that $\beta^{-1}(F^\circ) = \bigsqcup G_R^\circ$ consists of the
union of the interiors of boundary faces $G_R \in \cF([M; \cR])$ corresponding
to $R \in \cR \rst W(F)^\vee$ with $R^\circ \subset (W(F)^\vee)^\circ$, with
$\codim(G_R) = \dim(R)$ and $G_R \leq G_{R'} \iff R' \leq R$. 
By Proposition~\ref{P:b-maps_stratified}, for each $G \in \cF([M; \cR])$ there
is a unique $F \in \cF(M)$ with $\beta(G^\circ) \subset F^\circ$, so taking the union
over $F \in \cF(M)$, it follows that $\cP_{[M; \cR]} \cong \cR$.

The lifting of interior b-maps follows from the local version.
Indeed, given $f : N \to M$, we may refine the atlas on $N$ so that each chart
$O_b \subset N$ has image in some $O_a \subset M$, and then the lifted map is
given by patching together the lifted maps $\wt f : O_b \to U_a \subset [M;
\cR]$, which are necessarily compatible by the uniqueness in
Lemma~\ref{L:local_b-map_lift}.

The uniqueness of $[M; \cR]$ up to diffeomorphism follows from this, since
another choice of atlas leads to a space $[M; \cR]'$ and a unique
diffeomorphism $[M; \cR]' \cong [M; \cR]$.
\end{proof}

\subsection{Blow-up and pull back} \label{S:blow-up_pullback}

We recall one of the main results from \cite{joyce}, generalizing a similar
result for manifolds with ordinary corners in \cite{KM}.
Suppose $f : Y \to M$ and $g : Z \to M$ are interior b-maps of manifolds with
generalized corners. The maps are said to be {\bf b-transverse} if 
\[
	\bd f_\ast \oplus \bd g_\ast : \bT_y Y\oplus \bT_zZ \to \bT_x M
\]
is surjective for all $(y,z) \in Y \times Z$ such that $f(y) = g(z) = x$.

\begin{thm}[\cite{joyce}, Thm.~4.27]
If $f : Y \to M$ and $g : Z \to M$ are b-transverse maps, then the fiber
product (pull back) $Y \times_M Z$ exists in the category of manifolds with
generalized corners and interior b-maps. More explicitly,
\[
	Y \times_M Z = \ol{\set{(y,z) \in Y^\circ\times Z^\circ : f(y) = g(z)}} \subset Y\times Z
\]
admits a canonical structure of a manifold with generalized corners, with
respect to which it satisfies the following universal property: 
If $k : N \to Y$ and $l : N \to Z$ are interior b-maps such that $f\circ k =
g\circ l$, then there exists a unique interior b-map $h : N \to Y\times_M Z$ such
that $k = \pr_1\circ h$ and $l = \pr_2 \circ l$:
\[
\begin{tikzcd}[row sep=small, column sep=small]
	N \ar[dashed]{dr}{h} \ar[bend left]{drr}[above right]{l} \ar[bend right] {ddr}[below]{k}& & 
	\\ & Y\times_M Z \ar{r}[below]{\pr_2} \ar{d}{\pr_1} & Z \ar{d}{g}
	\\ & Y \ar{r}{f} & M.
\end{tikzcd}
\]
\label{T:fiber_products}
\end{thm}

\begin{prop}
A blow-down map $\beta : [M; \cR] \to M$ has the property that
\[
	\bd \beta_\ast : \bT_x [M;\cR] \stackrel \cong \to \bT_{\beta(x)} M
\]
is an isomorphism for all $x \in [M; \cR]$.
\label{P:b-etale}
\end{prop}
\begin{proof}
The property is local, so it suffices to consider the case $\beta : X_{R^\vee}
\to X_P$ for a maximal $R$ in a refinement of $P^\vee$.  By
Proposition~\ref{P:bT_local}, $\bT_x X_{R^\vee} \cong \Hom(R^\vee; \bbR) =
\Hom\big((R^\vee)^\gp; \bbR\big)$ and $\bT_{\beta(x)} X_P \cong \Hom(P; \bbR) =
\Hom(P^\gp; \bbR)$, and the linear map between them is generated by the
homomorphism $P^\gp \to (R^\vee)^\gp$ determined by duality from $R \to
P^\vee$. Since the latter is injective with $\dim(R) = \dim(P^\vee)$, the
former is an isomorphism.
\end{proof}

It follows that $\beta : [M; \cR] \to M$ is b-transverse to {\em any} interior
b-map $f : Y \to M$; in particular, the pull back
\[
	Y \times_M [M; \cR] \to Y 
\]
is a well-defined manifold for every map $f : Y \to M$.

On the other hand, fiber products exist in the category of monoidal complexes,
and the pull back of a refinement is a refinement \cite{KM}. 
In particular, given a refinement $\psi: \cR \to \cP_M$ and a map $f : Y \to
M$, we have a commutative diagram
\[
\begin{tikzcd}[row sep=small, column sep=small]
	\cP_Y \times_{\cP_M} \cR \ar{r} \ar{d} & \cR \ar{d}{\psi}
	\\ \cP_Y \ar{r}{f_\natural} & \cP_M
\end{tikzcd}
\]
of monoidal complexes, in which the vertical arrows are refinements.

\begin{thm}
Let $[M; \cR]$ be the blow-up of a manifold $M$ with respect to a refinement
$\cR \to \cP_M$, and let $f : Y \to M$ be an interior b-map. Then there is a canonical
diffeomorphism
\begin{equation}
	Y\times_M [M; \cR] \cong [Y; \cP_Y\times_{\cP_M} \cR]
	\label{E:fib_prod_diffeo}
\end{equation}
between the pull back of $[M; \cR]$ over $Y$ and the blow-up of $Y$ by the
refinement $\cP_Y \times_{\cP_M} \cR$. In other words, blow-ups pull back under
arbitrary interior b-maps.
\label{T:blow-up_pullback}
\end{thm}
\begin{proof}
Denote the refinement $\cP_Y\times_{\cP_M} \cR$ by $\cR' \to \cP_Y$.  Suppose
$N$ is a manifold with maps to to $Y$ and $[M; \cR]$ forming a commutative
square with $M$, thus inducing a commutative square of complexes:
\[
\begin{tikzcd}[row sep=small, column sep=small]
	N \ar{r} \ar{d} & {[M; \cR]} \ar{d}
	\\ Y \ar{r} & M,
\end{tikzcd}
\qquad
\begin{tikzcd}[row sep=small, column sep=small]
	\cP_N \ar{r} \ar{d} & \cR \ar{d}
	\\ \cP_Y \ar{r} & \cP_M.
\end{tikzcd}
\]
Then $\cP_N$ factors uniquely through $\cR'$ by the universal property of the
fiber product of complexes, and from the universal property of blow-up it
follows that $N$ factors through a unique map to $[Y; \cR']$. 
In other words, the manifold $[Y; \cR']$ satisfies the same universal property
as the fiber product $Y \times_M [M; \cR]$, which is unique up to canonical
diffeomorphism.
\end{proof}

\begin{rmk}
In the language of algebraic geometry, blow-down maps are {\em stable under base change}.
\end{rmk}

\subsection{Commentary} \label{S:commentary}
The assumption that boundary faces of manifolds are embedded is necessary if
one wants to work with monoidal complexes, as we have done. 
Indeed, the embeddedness assumption was used in
Proposition~\ref{P:pre_mon_cmplx} to obtain the triviality of the bundles $\bM
F$; without this assumption it is straightforward to construct examples where
the $\bM F$ are not trivial.
Moreover, even if the $\bM F$ are trivial, so that one still obtains a complex
$\cP_M$, it may not be possible to realize a refinement by blow-up, i.e., the
statement that $\cP_{[M; \cR]} \cong \cR$, which depends on the embeddedness
assumption, may fail.

To illustrate this last point, consider the {\em teardrop} (c.f. \cite{joyce},
Example~2.8)
\[
	M = \set{(x,y) \in \bbR^2 : 0 \leq x,\ y^2 \leq x^2 - x^4}.
\]
This is a 2-dimensional manifold with (ordinary) corners having a single
codimension 2 boundary face at the origin and a single, self-intersecting
boundary hypersurface. Its monoidal complex is
\begin{equation}
\cP_M : 
\begin{tikzcd}[column sep=small]
	\set 0 \ar[rightharpoonup]{r} & \bbN 
	 \ar[rightharpoonup,bend left]{r} \ar[rightharpoonup,bend right]{r} & \bbN^2
\end{tikzcd}
	\label{E:P_teardrop}
\end{equation}
where the single object $\bbN$ is identified with both of the faces of
$\bbN^2$. 
By contrast, the complex $\cP_{\bbN^2}$ is
\[
	\cP_{\bbN^2} :
\begin{tikzcd}[column sep=small, row sep=tiny]
	& \bbN \ar[rightharpoonup]{dr} & \\
	\set 0 \ar[rightharpoonup]{ur} \ar[rightharpoonup]{dr} & & \bbN^2 \\
	& \bbN \ar[rightharpoonup]{ur} &
\end{tikzcd}
\]
It is easy to see that there are no injective morphisms $\cP_M \to
\cP_{\bbN^2}$, while there is an obvious morphism $\cP_{\bbN^2} \to \cP_M$
given by the identity on each $\bbN^n$.
This latter morphism is a refinement, and the construction of $[M;
\cP_{\bbN^2}]$ given above goes through since it is completely local.  However,
since the only morphisms in the refinement are identities, we just recover $M$
again, i.e., $[M; \cP_{\bbN^2}] \cong M$, but $\cP_M \not\cong \cP_{\bbN^2}$.

To work with such spaces then, it is necessary to give up the complex $\cP_M$
in favor of a more complicated object. 

A {\bf monoidal space},
$(Y, \cM_Y)$, as defined by Kato \cite{kato} (see also \cite{GM}), is a
topological space $Y$ equipped with a sheaf $\cM_Y$ of sharp monoids, and a
morphism $f: (X,\cM_X) \to (Y,\cM_Y)$ is a continuous map $f : X \to Y$ with a
morphism $f^\ast\cM_Y \to \cM_X$ of sheaves.
A manifold $M$ with generalized corners 
admits the structure of a monoidal space, whose underlying topological
space is $M$, and whose sheaf of monoids is the sharpening $\cM_M =
\B_M^\sharp$ of the sheaf of b-functions.
%
This sheaf has the property that if $O$ meets a unique face $F$ of maximal
codimension, then $\B^\sharp_M(O) = W(F)$. 

In fact, the monoidal complex of a manifold with embedded boundary faces is
essentially equivalent to Kato's notion of the `fan' associated to certain
logarithmic schemes \cite{kato}. A {\bf fan} is a monoidal space locally
isomorphic to the `affine' model space $(\Spec(P),\cM_P)$. Here $\Spec(P) =
\set{F : F \leq P}$ is the set of faces (equivalently, prime ideals) of a
monoid $P$ equipped with the (non-Hausdorff) Zariski topology generated by open
sets $U_p = \set{F : p \in F}$ for $p \in P$, and $\cM_P$ is the sheaf of sharp
monoids whose stalk at $F \in \Spec(P)$ is the monoid $\cM_{P,F} = P/F$. 
(The concept of a fan is summed up succinctly by the analogy fan : sharp monoid
:: scheme : ring.) 
In contrast to a general monoidal space, a fan consists of a small (typically
finite) number of points; indeed, there is a bijection between the affine open
sets of a fan and its points (c.f.  Lemma~4.6, \cite{abram}). 
Certain sufficiently nice logarithmic schemes $(X,\cM_X,\cO_X)$ (analogous to
our manifolds with embedded boundary faces) are associated to a canonical fan
$F$ via a morphism $(X,\cM_X^\sharp) \to F$ which essentially collapses various
strata (analogous to our interiors of boundary faces) down to points.
In this analogy, the fan associated to a manifold with embedded boundary faces
has a single point for each stratum of \eqref{E:boundary_stratn} and is
equipped with a non-Hausdorff topology (encoding the inclusion relations
between boundary faces) and a sheaf obtained from the dual sheaf
$(\B^\sharp_M)^\vee$; in particular its monoids are dual to those in the
complex $\cP_M$.

That general manifolds (without embedded boundary faces) do not admit monoidal
complexes can be compared to the fact that not all logarithmic schemes admit
fans \cite{abram}. 
To define blow-up for manifolds in general, it should still be possible to
explicitly patch together the local constructions in \S\ref{S:blow-up_models}
for a suitable notion of refinement of the monoidal space $(M,
\B_M^\sharp)$. Indeed, this is the approach taken by \cite{GM}, though their
approach is rather abstract. Alternatively, it may be possible to work with
some kind of intermediate object which is simpler than $(M,\B_M^\sharp)$ but
more complicated than $\cP_M$ (compare the notion of an `Artin fan'
\cite{abram1, abram} in logarithmic algebraic geometry). We leave this for a
future work.

\bibliography{gbmwgc}

\providecommand{\bysame}{\leavevmode\hbox to3em{\hrulefill}\thinspace}
\providecommand{\MR}{\relax\ifhmode\unskip\space\fi MR }
\providecommand{\MRhref}[2]{%
  \href{http://www.ams.org/mathscinet-getitem?mr=#1}{#2}
}
\providecommand{\href}[2]{#2}
\begin{thebibliography}{10}

\bibitem{abram}
D.~Abramovich, Q.~Chen, M.~Ulirsch, and J.~Wise, \emph{Skeletons and fans of
  logarithmic structures}, arXiv:1503.04343, 2015.

\bibitem{abram1}
D.~Abramovich and J.~Wise, \emph{{Invariance in logarithmic Gromov-Witten
  theory}}, arXiv:1306.1222, 2013.

\bibitem{GM}
D.~Gillam and S.~Molcho, \emph{Log differentiable spaces and manifolds with
  corners}, arXiv:1507.06752, 2015.

\bibitem{joyce}
D.~Joyce, \emph{A generalization of manifolds with corners}, arXiv:1501.00401,
  2015.

\bibitem{kato}
K.~Kato, \emph{Toric singularities}, American Journal of Mathematics
  \textbf{116} (1994), no.~5, 1073--1099.

\bibitem{KKMS}
G.~Kempf, F.~Knudsen, D.~Mumford, and B.~Saint-Donat, \emph{Toroidal embeddings
  {I}}, Lecture notes in mathematics, vol. 339, Springer Verlag, 1973.

\bibitem{KM}
C.~Kottke and R.B. Melrose, \emph{{Generalized blow up of corners and fiber
  products}}, Transactions of the AMS \textbf{367} (2015), no.~1, 651--705.

\bibitem{melrose}
R.B. Melrose, \emph{{Differential analysis on manifolds with corners}}, In
  preparation, partially available at {\tt
  http://math.mit.edu/\textasciitilde{}rbm/book.html}.

\bibitem{melroseaps}
\bysame, \emph{{The Atiyah-Patodi-Singer index theorem}}, Research Notes in
  Mathematics, vol.~4, AK Peters, Ltd, Wellesley MA, 1993.

\bibitem{ogus}
A.~Ogus, \emph{Lectures on logarithmic algebraic geometry}, Notes available at
  {\tt
  http://math.berkeley.edu/\textasciitilde{}ogus/preprints/log\_book/logbook.pdf}
  (2006).

\bibitem{pflaum}
Markus~J. Pflaum, \emph{Analytic and geometric study of stratified spaces},
  Lecture Notes in Mathematics, vol. 1768, Springer Verlag, 2001.

\end{thebibliography}
\bibliographystyle{amsplain}

\end{document}